\documentclass[12pt,a4paper,oneside,reqno]{amsart}

\usepackage[a4paper,margin=1.05in]{geometry}
\usepackage[utf8]{inputenc}
\usepackage{amsmath,amsthm,amssymb}
\usepackage{mathrsfs} 
\usepackage{graphicx, xcolor}
\usepackage{caption} 
\usepackage{subfiles}
\usepackage{newverbs} 
\usepackage{array} 
\usepackage{braket}

\usepackage{tikz}
\usetikzlibrary{arrows,decorations.pathreplacing,decorations.markings,calligraphy}

\usepackage[colorlinks,citecolor=blue!40!black,linkcolor=green!40!black,pagebackref=true]{hyperref}
\usepackage{cleveref}

\usepackage{enumitem}

\newlist{PROPenum}{enumerate}{1} 
\setlist[PROPenum]{label=(\roman*),ref=\thePROP\,(\roman*)}
\crefname{PROPenumi}{Proposition}{Propositions}

\newlist{itemcases}{enumerate}{1}
\setlist[itemcases]{
label=\textbf{Case~(\arabic*).},
ref=(\arabic*),
leftmargin=*
}

\crefname{itemcasesi}{Case}{Cases}
\Crefname{itemcasesi}{Case}{Cases}

\newlist{subitemcases}{enumerate}{1}
\setlist[subitemcases]{
label=\textbf{Case~3\alph*.},
ref=3\alph*,
leftmargin=*
}

\crefname{subitemcasesi}{Case}{Cases}
\Crefname{subitemcasesi}{Case}{Cases}

\usepackage{comment}

\newverbcommand{\CMRverb}{\tiny\color{blue}}{}
\newverbcommand{\GRverb}{\tiny\color{teal}}{}
\newverbcommand{\STverb}{\tiny\color{cyan}}{}
\newverbcommand{\Pverb}{\tiny\color{violet}}{}
\newverbcommand{\Averb}{\tiny\color{brown}}{}

\theoremstyle{plain}
\newtheorem{THM}{Theorem}[section]

\newtheorem{PROP}[THM]{Proposition}
\newtheorem{LEM}[THM]{Lemma}

\newtheorem*{CLAIME}{Claim}

\newtheorem{QUEST*}{Question}

\theoremstyle{definition}

\newtheorem{RMK}[THM]{Remark}

\DeclareMathOperator{\Leb}{Leb}
\DeclareMathOperator{\dev}{dev}

\newcommand{\ie}{\text{i.e. }}
\newcommand{\dfn}{\mathrel{\mathop:}=}

\newcommand{\C}{\mathbb{C}}

\newcommand{\FF}{\mathcal{F}}

\newcommand{\Q}{\mathbb{Q}}
\newcommand{\R}{\mathbb{R}}

\newcommand{\Z}{\mathbb{Z}}
\newcommand{\N}{\mathbb{N}}
\newcommand{\E}{\mathcal{E}}
\renewcommand{\d}{\mathrm{d}}
\newcommand{\area}{\mathrm{area}}
\newcommand{\tremspace}{\mathcal{T}}
\newcommand{\trem}{\mathrm{trem}}
\newcommand{\hol}{\mathrm{hol}}
\newcommand{\holx}{\mathrm{hol}^{(x)}}
\newcommand{\holy}{\mathrm{hol}^{(y)}}
\renewcommand{\H}{\mathcal{H}}
\DeclareMathOperator{\Rel}{\mathrm{Rel}}
\DeclareMathOperator{\length}{length}

\renewcommand{\bar}{\overline}
\renewcommand{\|}{|\!|} 

\title[Some real {Rel} trajectories in $\mathcal{H}(1,1)$ that are not recurrent]{Some real {Rel} trajectories in $\mathcal{H}(1,1)$ that are not recurrent}
\author[Carlos Ospina]{Carlos Ospina}
\address{School of Mathematical Sciences\\
Tel Aviv University\\
Tel Aviv\\
69978\\
Israel}
\email{ospina.math@icloud.com}

\begin{document}

\begin{abstract}
We study the real Rel orbits of certain translation surfaces in the stratum $\H(1,1)$, specifically surfaces that are tremors of the locus of branched double covers of tori. We give necessary and sufficient conditions on tremors of a surface so that the real Rel orbit is recurrent. As a consequence, we are able to provide explicit examples of trajectories of real Rel that are not recurrent.
\end{abstract}

\maketitle

\section{Introduction}
\textcolor{black}{For $g \geq 1$ and nonnegative integers $r_1,\dots,r_k$ satisfying
\[
\sum_{i=1}^k r_i = 2g-2,
\]
the stratum $\H(r_1,\dots,r_k)$ parametrizes translation surfaces of genus $g$ with $k$ marked singularities, where the $i$th singularity has cone angle $2(r_i+1)\pi$. In this work we \emph{distinguish} each marked singularity. }

If $k \geq 2$, there is a foliation of dimension $k-1$ called real Rel ---the formal definitions are in \Cref{sec:Rel}. Two sufficiently close surfaces in a leaf are obtained from each other by a surgery that only modifies the relative horizontal positions of the singularities.
When $k=2$, the leaves of this  foliation have dimension one and determine a directional flow. We denote by $\Rel_t(\cdot)$ the real Rel flow in this situation.

Let $\E \subset \H(1,1)$ be the subset of translation surfaces that are two identical tori glued along a slit. See \Cref{Section:Preliminaries} for more details.
If $M \in \E$ is horizontally aperiodic and non-minimal, then we can represent $M$ by a slit construction $T_1\#_IT_2$, where $T=T_1=T_2$ is a horizontally minimal torus and $I$ (the slit) is a horizontal segment embedded on $T$. 
\textcolor{black}{In this case, $\Rel_t M$ for $t \in (0,\infty)$ makes the slit longer while keeping it horizontal.} 
We say $N \in \H(1,1)$ is \emph{recurrent} if there exists a sequence $t_i \to \infty$ such that $\lim_{i \to \infty} \Rel_{t_i}N = N.$ 

\textcolor{black}{%
We first give two basic examples. Suppose that $T=\R^2/\Z^2$ and that $I$ is a vertical segment of length $h\in(0,1)$. The resulting surface $M$ has a horizontal decomposition into three cylinders $C_1,C_2,C_3$. The height, circumference and twist of these cylinders are $(h_i,c_i,\tau_i)=(1-h,1,0)$, $i=1,2$,
and $(h_3,c_3,\tau_3)=(h,2,0)$. 
The real Rel deformation preserves this cylinder decomposition, keeping the heights and circumferences fixed, but changes the twist parameters at constant speed.
It is a fact of the translation structure that, for such a cylinder decomposition with fixed heights and circumferences, the surface is unchanged precisely when the twist vector satisfies
\[
(\tau_1,\tau_2,\tau_3)=(0,0,0)
\mod (c_1,c_2,c_3).
\]
Let $x_1$ and $x_2$ be the endpoints of the slit $I$, where $x_1$ is the bottom endpoint and $x_2$ is the top endpoint. With the convention that real Rel moves $x_2$ relative to $x_1$, the corresponding twist parameters of $\Rel_t M$ are $(-t,-t,t)$.
Since $(c_1,c_2,c_3)=(1,1,2)$, we have
\[
(-2,-2,2)=(0,0,0) \mod (1,1,2).
\]
Therefore $\Rel_2 M=M$ and the positive real Rel trajectory of $M$ is recurrent.
On the other hand, suppose that $I$ is horizontal. Then, after a finite time $t_0>0$, the two singularities collide, so $\Rel_{t_0}M$ is no longer contained in $\mathcal{H}(1,1)$. Thus, the real Rel trajectory diverges and is not recurrent.
}
Let $T=\R^2/\Lambda$ be a torus and let $(a,b)$ and $(c,d)$ be two generators of $\Lambda$. Our main result is the following.
\begin{THM}
\label{thm:mainthm}
Let $M \in \E$ be a horizontally aperiodic and non-minimal surface and let $\beta \in \tremspace_M$ be a tremor such that $\beta \not= l\d y$ for all $l \in \R\backslash\{0\}$. Then, the surface $\trem_\beta M$ is recurrent
if and only if $\frac{b}{d} \not \in \Q$ is well approximable.
\end{THM}

Tremors were introduced in \cite{2020tremors} to describe the action of the horocycle flow and can be considered as deformations, given transverse invariant measures of the horizontal foliation. The real Rel flow commutes with tremors and we use this property to exhibit different behaviors of the real Rel flow.

The real Rel foliation is a subfoliation of a larger foliation, called the Rel foliation. Other names for the Rel foliation that appear in the literature are the kernel foliation, the absolute period foliation and the isoperiodic foliation.
Since the present work focuses on real Rel trajectories rather than on the full Rel foliation, we would like to refer the reader to \cite{CDF,HoroDynamics,Ygouf2022,Hamenstadt,McMullenIsoperiodicPpr,McMullen2007II} and the references therein.
The topological and measure-theoretic properties of the Rel foliation and of the real Rel flow have been an active area of research in recent years. 
\textcolor{black}{%
For instance, Hooper and Weiss~\cite{HooperWeiss} showed that, for $g \geq 3$, there exists a surface of genus $g$ that has a divergent real Rel trajectory in $\H(g-1,g-1)$. That is, the trajectory leaves every compact set and never returns. We note that there are many divergent real Rel trajectories arising from horizontal saddle connections: in such examples, the trajectory is defined only for a finite time, after which singularities collide;  see the example before \Cref{thm:mainthm}. The novelty of~\cite{HooperWeiss} is that it gives a different mechanism for divergence, producing a trajectory that is defined for all positive times but is nevertheless divergent.\\
}
\textcolor{black}{%
On every connected component of a stratum of translation surfaces with more than one singularity, the real Rel flow preserves the natural Lebesgue measure, known as the \emph{Masur--Veech} measure. Thus, by the classical Poincar\'e recurrence theorem, almost every real Rel orbit is recurrent. This gives a measure-theoretic contrast with the divergent examples discussed above.\\
There are also results concerning dense and equidistributed real Rel orbits. Winsor~\cite{winsor2022dense} gives concrete examples of dense real Rel orbits in connected components of strata. More recently, in~\cite{chaika2023ergodic}, it was proved that the real Rel flow is mixing of all orders with respect to Masur--Veech measure, and in particular ergodic.
}
The complete classification of orbit closures of real Rel is still an open problem, but the present work shows that there are more flexible behaviors.
A consequence of \Cref{thm:mainthm}
is the following.

\begin{THM}
\label{thm:nongenbehav}
There are real Rel trajectories in $\H(1,1)$ that are non-divergent but not recurrent.
\end{THM}

The organization of this paper is as follows. \Cref{sec:BigPrelim} contains the definitions and setting for the statement of \Cref{thm:mainthm}. 
In \Cref{sec:RelRec}, we state \Cref{prop:recurrencecase}, which by explicit arguments proves the {\it if} in \Cref{thm:mainthm}.
In \Cref{sec:proofoflemmaboundedarea}, we will prove the {\it only if} direction of \Cref{thm:mainthm} using \Cref{thm:NonRecurrenceofRel}.

\section*{Acknowledgements}
The author would like to thank Jon Chaika for many helpful conversations and comments throughout the development of this work. The author also thanks the anonymous referee for their thoughtful comments which improved the manuscript substantially.

\section{Preliminaries}
\label{sec:BigPrelim}
In this section, we set up the basic definitions and notation used throughout the paper. These definitions can be found in \cite{HoroDynamics,2020tremors}; for other helpful references, see \cite{Zorich,MasurTabachnikov,claynotes} and the references therein. Although the main results of this paper concern the stratum $\H(1,1)$ and, more specifically, the locus $\E\subset\H(1,1)$, we begin by recalling some general background on translation surfaces, strata, real Rel, and tremors. We will indicate explicitly when the discussion specializes.

\subsection{Background on translation surfaces}
\textcolor{black}{%
A \emph{translation surface} $M$ is a compact, oriented surface obtained from a finite collection of polygons $P_1,\dots,P_m$ by identifying pairs of edges by translations. More precisely,
\[
M=\bigcup_{i=1}^m P_i \big\slash\sim.
\]
We orient each polygon counterclockwise. The equivalence relation $\sim$ identifies each edge $e_i$ with a unique edge $e_j$ that is parallel, has the same length and has the opposite orientation induced by the counterclockwise orientations of the polygons.\\ 
After embedding the polygons in $\mathbb{C}$, there is a translation $z\mapsto z+c_{ij}$ sending the edge $e_i$ to the edge $e_j$. Thus, for $z\in e_i$ and $z'\in e_j$, we have
\[
z\sim z'
\quad\text{if and only if}\quad
z'=z+c_{ij}.
\]
Let $\Sigma:=\Sigma_M\subset M$ be the set of equivalence classes of vertices of the polygons. The total angle around each point of $\Sigma$ is an integer multiple of $2\pi$. Thus, if $p_i\in\Sigma$, the cone angle at $p_i$ is of the form
\[
2\pi(r_i+1)
\]
for some integer $r_i\geq 0$. Points with $r_i>0$ are called \emph{singularities}, while points with $r_i=0$ are called \emph{marked points}. If $g$ is the genus of $M$ and $k=|\Sigma|$, then
\[
\sum_{i=1}^k r_i = 2g-2.
\]
}

The polygons defining a translation surface induce a family of charts on $M$, whose transition maps are translations in $\C$. 
\textcolor{black}{%
We refer the reader to \cite[Chapter~2]{MasurAthreya} for a formal description. 
A translation atlas is a tuple
\[
\bigl(X,\{U_\alpha\}_{\alpha\in\mathcal A} \cup \{V_p\}_{p \in \Sigma},
\{\varphi_\alpha\}_{\alpha\in\mathcal A}\cup \{\psi_p\}_{p \in \Sigma}\bigr),
\]
where $X$ is a compact topological surface, $\Sigma \subset X$ is a finite set, $\{U_\alpha\}_{\alpha\in\mathcal A}$ is an open cover of $X\setminus\Sigma$, and 
$\varphi_\alpha:U_\alpha\to\mathbb C$ 
are charts such that the transition maps $\varphi_\beta\circ\varphi_\alpha^{-1}$ 
are restrictions of translations $z\mapsto z+c_{\alpha\beta}$. These charts are called regular charts.
The set $V_p \subset X$ is a neighborhood of $p\in\Sigma$, the local structure is modeled on a branched cover $\psi_p:V_p \to \mathbb{D}_{\rho_p} $ of a small disk $\mathbb{D}_{\rho_p}\subset \mathbb{C}$. The degree of the branched cover is $r_p+1$. The transition maps between regular and singular charts are also restrictions of translations.
Therefore, a translation surfaces $M$ is an equivalence class of translation atlases. The equivalence relation is defined via homeomorphisms of the underlying topological surface which  preserve the translation structure. Conversely, an equivalent class of translation atlas admits a polygonal decomposition. This follows by the Delaunay polygonal decomposition, see \cite[Section 3.6]{MasurAthreya}.
}


\subsection{Horizontal flow}
\label{sec:HorFlow}

The equivalence class of translation atlases determines singular foliations on $M$. 
\textcolor{black}{%
Indeed, in a neighborhood of any point $x\in M\setminus\Sigma$, the transition maps of a translation atlas sends oriented straight segments to straight segments and preserve both orientation and lengths. Hence, for each $\theta\in [0,2\pi)$, the straight-line foliation of $\mathbb C$ in direction $\theta$ pulls back to a foliation on $M\setminus\Sigma$. Its leaves are locally the preimages, under translation charts, of straight lines in $\mathbb C$ in direction $\theta$.
This foliation extends to a singular foliation on $M$. If $p\in\Sigma$ has cone angle $2\pi(r_p+1)$, then, in a sufficiently small neighborhood of $p$, the foliation in direction $\theta$ has $r_p+1$ distinguished local branches meeting at $p$. These local branches are called the \emph{prongs} of the foliation at $p$.
We identify $\mathbb C$ with $\mathbb R^2$ in the canonical way. We will mainly consider two foliations on $M$: the \emph{horizontal foliation}, corresponding to the direction $\theta=0$, and the \emph{vertical foliation}, corresponding to the direction $\theta=\pi/2$. Thus, there is a well defined notion of moving up or down, left or right at every $x \in M \setminus \Sigma$.
Denote by $\d x$ and $\d y$ the real $1$-forms induced on $M$ by the standard coordinate forms on $\R^2$. Thus, $\d x$ measures horizontal displacement and $\d y$ measures vertical displacement. The vertical foliation on $M$ is defined by the equation $\d x=0$, while the horizontal foliation is defined by $\d y=0$.
}

The charts define a canonical measure $\Leb$ on $M$: a set $A \subset M \cap U_\alpha$ is measurable if and only if $\varphi_\alpha(A) \subset \R^2$ is measurable, and $\Leb(A)$ is equal to the Lebesgue measure of $\varphi_\alpha(A)$.

Let $\mathrm{U} \subset M$ be the union of the leaves of the horizontal foliation that do not contain singularities. We denote by $\FF^t: \mathrm{U} \to \mathrm{U}$ the horizontal flow; that is, $\FF^t$ sends $x\in \mathrm{U}$ to the point $\FF^t(x)$ on the same horizontal leaf at distance $|t|$ to the right or left of $x$, depending on the sign $\pm1$ of $t$.

We say that the horizontal flow is \emph{minimal} if all $\FF$-trajectories are dense. We say that $M$ is \emph{horizontally minimal} if $\FF$ is minimal. If a leaf $\gamma$ of the foliation in direction $\theta$ starts and ends at singularities, possibly distinct, and contains no singularities in its interior, then the path it determines is called a \emph{saddle connection}. The integral
\[
\int_\gamma \omega \in \R^2
\]
is called the \emph{holonomy} vector of $\gamma$.

\subsection{Strata and period coordinates}
\label{sec:stratum}

Let $M_1$ and $M_2$ be translation surfaces of type $(r_1,\dots,r_k)$. A \emph{translation equivalence} is a smooth orientation-preserving homeomorphism 
\[
\phi:M_1\to M_2
\]
that maps singularities to singularities and whose derivative is the identity matrix. 
Fix $g$, $k$ and $\{r_1,\dots,r_k\}$ with $\sum r_i = 2g-2$. We denote by $\H(r_1,\dots,r_k)$ the \emph{stratum}  of translation surfaces up to translation equivalence where the singularities are of type $(r_1,\dots,r_k)$. 

This work will consider two strata: $\H(0)$ and $\H(1,1)$. \textcolor{black}{%
These parametrize genus one translation surfaces with one marked point and genus two translation surfaces with two different singularities of cone angle $4\pi$, respectively. 
} 

Let $S$ be an oriented surface of genus $g$, and let $\Sigma$ be a set of $k$ distinct points $\xi_1,\dots,\xi_k$.
A \emph{marked translation surface} is a pair $(f,M)$, where $M$ is a translation surface of type $(r_1,\dots,r_k)$ and $f:S\to M$ is an orientation-preserving homeomorphism such that the marked points $\Sigma\subset S$ are mapped to the singularities of $M$.
Two marked translation surfaces $(f_1,M_1)$ and $(f_2,M_2)$ are equivalent if there exists a translation equivalence $\phi:M_1 \to M_2$ and the maps $\phi \circ f_1$ and $f_2$ are isotopic relative to the marked points $\Sigma$.
We denote the set of equivalence classes of marked translation surfaces of type $(r_1,\dots,r_k)$ by $\H_m(r_1,\dots,r_k)$ and call it \emph{marked stratum.} 

Recall that in \Cref{sec:HorFlow}, a translation surface $M=(X,\omega) \in \H,$  determines two differential forms $\d x$ and $\d y$. These determine a cohomology class $\hol_M \in H^1(X,\Sigma;\R^2)$ by
\begin{equation*}
\hol_M(\gamma) = (\holx(\gamma),\holy(\gamma)),
\end{equation*}
where $\holx(\gamma)= \int_\gamma \d x$ and $\holy(\gamma)= \int_\gamma \d y$. Note that when we write $\holx_M$ instead of $\holx$, we are emphasizing that the cohomology class is determined by the translation surface $M$. We can do the same for $\holy_M$ instead of $\holy$.\\ 
This determines the map 
\begin{equation}
\begin{split}
\dev:\,
&
\H_m \to H^1(S,\Sigma;\R^2)\\
&
[(f,M)] 
\mapsto
f^*(\hol_M)
\end{split}
\end{equation}
such that, for every $\gamma\in H_1(S,\Sigma;\Z)$,
\[
    f^*(\hol_M)(\gamma)=\hol_M(f_*(\gamma)).
\]
We denote $\hol_{(f,M)}\dfn \hol_M(f_*(\gamma)),$ likewise $\holx_{(f,M)}$ and $\holy_{(f,M)}$.
\textcolor{black}{%
This map is a local homeomorphism and it provides the structure of an affine manifold to $\H_m$, see \cite[Proposition 2.1]{HoroDynamics}.
}

\subsubsection{Orbifold structure of $\H$}
The pure mapping class group of $(S,\Sigma)$ denoted by $\operatorname{Mod}(S,\Sigma)$ is the group of isotopy classes of orientation-preserving homeomorphisms of $S$ that fix $\Sigma$ pointwise.
Given $h \in \operatorname{Mod}(S,\Sigma)$ define the action on $\H_m$ by  $h[(f,M)] \dfn [(f\circ h^{-1},M)]$. 
The action is transitive on the classes of marking maps of a given translation surface $M$. 
\textcolor{black}{Hence, we obtain the identification of the strata $\H$ with the quotient
$\H_m/\operatorname{Mod}(S,\Sigma)$. Denote by $\varpi:\H_m \to \H$ the quotient map. }\\
Consider a polygonal representation of a translation surface $M$. By adding edges between vertices of the polygons, we obtain a triangulation whose edges are geodesic segments and whose set of vertices is $\Sigma$.
Let $f$ be a marking of $M$, and let $\tau$ be the triangulation of $S$ obtained as the preimage under $f$ of a geodesic triangulation of $M$. The map $\dev$ assigns to $[(f,M)]$ the cohomology class $f^*(\hol_{M})$. In particular, this class assigns to the edges of $\tau$ nonzero vectors in $\R^2$ and maps the nondegenerate triangles of $\tau$ to nondegenerate triangles in $\R^2$.\\
Let $U_\tau\subset H^1(S,\Sigma;\R^2)$ be the set of cohomology classes that map the nondegenerate triangles of $\tau$ to nondegenerate triangles in $\R^2$ with the same orientations as those determined by $f^*({\hol}_M)$. Then $U_\tau$ is open.\\
Following \cite[Section~2.2]{2020tremors}, define a map
\[
\Psi_\tau:U_\tau \to \H_m,
\qquad
\beta \longmapsto 
[(\operatorname{Aff}_\beta\circ f,M_\beta)]
\]
where $M_\beta$ is the translation surface obtained by gluing the triangles determined by $\beta$. The map  $\operatorname{Aff}_\beta:M\to M_\beta$ is the piecewise affine map that sends each triangle of $M$ to the corresponding triangle in $M_\beta$.
Denote  $V_\tau=\Psi_\tau(U_\tau)$. By construction, the map
\[
\Phi_\tau:V_\tau\to U_\tau,
\qquad
\Phi_\tau([(g,N)])=\dev([g,N]),
\]
is a local inverse. We endow $\mathcal H_m$ with the topology induced by the charts $\{(V_\tau,\Phi_\tau)\}$. 
\textcolor{black}{
The transition maps of the charts $\{(V_\tau,\Phi_\tau)\}$ are restrictions of affine maps on $H^1(S,\Sigma;\R^2)$. Therefore, the tangent bundle $T(\H_m)$ is of the form $\H_m \times H^1(S,\Sigma;\R^2)$.}

Denote by $\Gamma_M$ the group of translation equivalences of $M$, and let $\mathcal{G}_{(f,M)}$ be the subgroup of $\operatorname{Mod}(S,\Sigma)$ consisting of the mapping classes of maps of the form
\[
f^{-1} \circ g \circ f,
\qquad 
g \in \Gamma_M.
\]
The family of charts
\[
\{
(U_\tau,\mathcal{G}_{(f,M)},\varpi \circ \Psi_\tau)\}
\]
defines an orbifold structure on $\H$. Each orbifold chart consists of the open set $U_\tau$ and the map
\[
\varpi\circ\Psi_\tau:U_\tau\to\mathcal H,
\]
which is $\mathcal G_{(f,M)}$-invariant. After replacing $U_\tau$ by a sufficiently small open subset $U'_\tau\subset U_\tau$, the induced map
\[
U'_\tau/\mathcal G_{(f,M)}\to\mathcal H
\]
is injective. 
We will denote the tangent space of $M \in \H$ by $T_M\H$. The real dimension of $\H_m$ and $\H$ is $2(2g+k-1)$. For more details see \cite[Section~2.3]{2020tremors}.

\subsection{The \texorpdfstring{$\operatorname{SL}(2,\R)$}{} action on strata}

Denote by $G=\operatorname{SL}
(2,\R)$,
and the subgroups $U = \{u_s\}_{s\in \R}$
and $A=\{g_s\}_{s \in \R}$ with 
\begin{equation*}
u_s = \begin{pmatrix} 1&s\\0&1\end{pmatrix}
\qquad \text{and} \qquad 
g_s = \begin{pmatrix} e^{s}&0\\0&e^{-s}\end{pmatrix}.
\end{equation*}
From now on, we assume that $G$ acts on $\R^2$ by linear maps.
Suppose $M$ is a translation surface constructed with the polygons $P_1,\dots, P_m$. For any $g \in G$, the translation surface $gM$ is defined by the polygons $gP_1,\dots,gP_m$ with the edges $ge_i$ and $ge_j$ glued whenever the edges $e_i$ and $e_j$ were glued in the construction of $M$.
The $G$-action is equivariant with respect to the holonomy map, \ie  
\begin{equation*}
\hol_{gM}=g\hol_{M}.
\end{equation*}
Analogously, there is a $G$-action on the marked stratum $\H_m$ that is $\varpi$-equivariant. We omit the details and refer the reader to \cite[Section 2.4]{2020tremors}.

\subsection{Rel foliation and real Rel}
\label{sec:Rel}
Following \cite{HoroDynamics},  consider the exact sequence 
\begin{equation*}
0
\to
H^0(X;\R^2)
\to
H^0(\Sigma;\R^2)
\to
H^1(X,\Sigma;\R^2)
\to
H^1(X;\R^2)
\to 
0,
\end{equation*}
where $Res:H^1(X,\Sigma;\R^2)
\to 
H^1(X;\R^2)$ is the restriction map since there is a natural inclusion $H_1(X;\R^2) \hookrightarrow H_1(X,\Sigma;\R^2)$. 
The kernel of $Res$ can be identified with $\mathfrak{R} = H^0(\Sigma)/H^0(X)$ and the vector space $H^1(X,\Sigma;\R^2)$ can be foliated by translates of $\mathfrak{R}$. 
Moreover, this foliation determines a foliation of $\H$, called \emph{Rel foliation.}
Since the dimension of $\mathfrak{R}$ determines the dimension of the leaves of the Rel foliation, their real dimension is $2(k-1).$ 
For example in $\H(1,1)$, since $k=2$ the leaves have real dimension 2.
Intuitively, two nearby translation surfaces are in the same leaf if the
holonomy of their simple closed geodesics is the same and the holonomy of saddle connections differs only in the horizontal component.
\textcolor{black}{
Consider the splitting
\begin{equation}
H^1(S,\Sigma;\R^2)=H^1(S,\Sigma;\R_x) \oplus H^1(S,\Sigma;\R_y)  
\label{eq:splitting}
\end{equation}
from writing $\R^2 = \R_x\oplus\R_y$.}
The splitting in \Cref{eq:splitting} determines the splitting of $\mathfrak{R}$  into $\mathfrak{R}_x\oplus\mathfrak{R}_y$ and similarly to the definition of the Rel foliation, we obtain two different sub-foliations of $\H$. 
The sub-foliation determined by $\mathfrak{R}_x$ is called \emph{real Rel} foliation and the foliation determined by $\mathfrak{R}_y$ is called the \emph{imaginary Rel} foliation. Both foliations have real dimension $k-1$. 
In the case of stratum $\H(r_1,r_2)$, the real Rel foliation determines a flow, the \emph{real Rel flow.} 
We denote this flow by $\Rel_t$. 

Given $M \in \H(1,1)$, the real Rel flow $t \mapsto \Rel_tM$ is defined for all $ t \in I_M \subset \R$, where $I_M$ is an open interval containing 0. Extending the flow to a larger interval or even $\R$ is related to potential collisions of singularities. For a formal explanation of this issue, see \cite[Theorem 6.1]{HoroDynamics}.
The following summarizes important properties of the real Rel flow.
\begin{PROP}
\label{thm:RelProperties}
Let $\H=\H(1,1)$ and $M\in \H$. Let $I \subset \R$ be an open interval containing 0. Then:
\begin{PROPenum}
\item \label{item:RelGequivariance} If $\Rel_tM$ is defined for all $t\in I$, then for  all $t \in I,$ and all $s \in \R$,
\begin{equation*}
g_s\Rel_t M = \Rel_{e^st}g_sM \quad \text{and} \quad 
u_s\Rel_tM = \Rel_t u_sM.
\end{equation*}

\item \label{item:flowproperty} If $\Rel_{s+t}M$, $\Rel_sM$, and $\Rel_tM$ are defined, then 
\begin{equation*}
\Rel_{s+t}M = \Rel_s\Rel_tM =\Rel_t\Rel_sM.
\end{equation*}

\item \label{item:RelCondition}
\textcolor{black}{
Let $\Lambda_M \subset \R^2$ be the set
\begin{equation*}
\Lambda_M = \{ \hol_M(\gamma): \gamma \text{ is a saddle connection from $\zeta_2$ to $\zeta_1$}\},
\end{equation*} 
then $\Rel_tM$ is well defined if and only if for all $v \in \Lambda_M$ and $0<x<t$ (or $ t<x<0$), 
\begin{equation*}
v-(x,0) \not = (0,0).
\end{equation*}
}
\end{PROPenum}
\end{PROP}

\Cref{item:RelGequivariance} follows from  \cite[Proposition 4.4]{HoroDynamics} and that the subgroup $AU$ leaves horizontal directions invariant.
\Cref{item:flowproperty} was proven in \cite[Proposition 4.5]{HoroDynamics}.  \Cref{item:RelCondition} follows from \cite[Corollary 6.2]{HoroDynamics}.

\subsection{Tremors}

\label{sec:tremors}


We refer \cite[Section 2.5]{2020tremors} for the following definitions.
A \emph{transverse measure} $\nu$ on a translation surface $M$ is a family $\{\nu_\gamma\}$ of Borel measures indexed by transverse curves to the horizontal foliation $\gamma$. Each measure $\nu_\gamma$ is finite, non-negative and supported on $\gamma$. These measures are invariant by isotopies preserving the horizontal foliation and if $\gamma'\subset \gamma$ 
then $\nu_{\gamma'}$ is the  restriction of $\nu_\gamma$ to $\gamma'.$  
We will consider only non-atomic transverse measures. 

A \emph{signed transverse measure} $\nu$ is a family of Borel measures $\{\nu_\gamma\}$ that satisfy the same properties of a transverse measure although $\nu_\gamma$ can assume negative values. By Hahn decomposition, we can uniquely represent $\nu_\gamma$ as the difference of two non-negative measures: $\nu_\gamma^+ - \nu_\gamma^-$. Thus we write $\nu = \nu^+ - \nu^-$. We say $\nu$ is non-atomic if $\nu^+$ and $\nu^-$ are non-atomic. 
In \cite[Section 2.3]{2020tremors}  it was shown that a (signed) transverse measure $\nu$ determines a cohomology class $\beta_\nu \in H^1(M,\Sigma;\R)$. We call a cohomology class a \emph{(signed) foliation cocycle} if it is determined in this way.

Let $M\in \H$ and $\d y$ be the differential form in \Cref{sec:stratum}. $\d y$ defines the transverse measure $\gamma \mapsto \int_\gamma \d y$. To simplify the notation, we write $\d y$ for this transverse measure and $\holy_M$ for the corresponding foliation cocycle.

We denote by $\tremspace_M \subset H^1(M,\Sigma;\R)$ the space of signed foliation cocycles of $M$  and $C^+_M \subset \tremspace_M$  the cone of foliation cocycles. Identifying $\R$ with $\R_x$, we can see $\tremspace_M$ is a subspace of
\begin{equation*}
H^1(M,\Sigma;\R^2) = H^1(M,\Sigma;\R_x) \oplus H^1(M,\Sigma;\R_y).
\end{equation*}

The following results relate transverse and invariant measures of the horizontal flow. See \cite[Proposition 2.3]{2020tremors}.
\begin{PROP}
\label{prop:InvMeasures}
For each non-atomic transverse measure $\nu$ there exists an $\FF$-invariant measure $\mu_\nu$ with $$\mu_\nu(A)=\nu(v)l(h) $$ for every isometrically embedded rectangle $A$ with one horizontal side $h$ and another side $v$ orthogonal to the horizontal direction, where $l$ is the Euclidean length. The map $\nu \mapsto \mu_\nu$ is a bijection between non-atomic transverse measures and $\FF$-invariant measures that assign zero measure to horizontal leaves. It extends to a bijection between non-atomic signed transverse measures and $\FF$-invariant signed measures assigning zero measure to horizontal leaves.
\end{PROP}

For $M \in \H$ and $\beta = \beta_\nu \in \tremspace_M$, let $\mu_\nu$ be the invariant measure in  $\Cref{prop:InvMeasures}.$ The \emph{signed mass} of $\beta$ is $L_M(\beta)\dfn\mu_\nu(M)$. In particular, if $\beta \in C^+_M$, then $L_M(\beta)\ge0$. Writing a signed transverse measure $\nu$ of the form $\nu^+-\nu^-$, then every $\beta \in \tremspace_M$ can be written  $\beta=\beta^+-\beta^-$ with $\beta^+,\beta^- \in C^+_M$. We denote the \emph{total variation} by $|L|_M(\beta) = L_M(\beta^+)+L_M(\beta^-).$ 
The signed mass $L_M(\beta)$ is also equal to evaluating the cup product $\holx_M \smile \beta$ on the fundamental class of $M$. This last observation implies that the map $(\beta,M) \mapsto L_M(\beta)$ is continuous.

Recall that for $M \in \H$, we denote the tangent space of $\H$ at $M$ by $T_M\H$. In \Cref{sec:stratum}, (and also in \cite[Section 2.2]{2020tremors}), $T_M\H$ is identified with $H^1(M,\Sigma;\R^2).$ Thus, we regard $\tremspace_M$ as a subspace of $T_M\H$.

The following summarizes the properties of the previous objects. Proofs can be found in \cite[Section 4]{2020tremors}.
\begin{PROP}The following hold:
\begin{PROPenum}
\item The set
\begin{equation*}
C^+_{\H,1}=\{(M,\beta):
\beta
\in C^+_M,L_M(\beta)=1\}
\end{equation*}
is a closed subset of $T\H$.
\item Suppose $M \in \H$  has area one and $\FF$ is uniquely ergodic. If $\{(M_n,\beta_n)\}$ is a sequence with $\beta_n \in C^+_{M_n}$ with $L_{M_n}(\beta_n)=1$ such that $M_n \to M$, then $\beta_n \to \holy_M$.
\end{PROPenum}
\end{PROP}

\subsubsection{Tremor map}
For ${\widehat M} \dfn [(f,M)] \in \H_m$, note that 
$C^+_{\widehat M} \dfn f^*(C^+_M)$ and 
$\tremspace_{\widehat M}\dfn f^*(\tremspace_M)$ 
are subspaces of $H^1(S,\Sigma;\R_x)$. 

For $\beta \in \tremspace_{\widehat M}$, we define $\trem_\beta {\widehat M} \dfn \theta(1) $ where $\theta(t)$ is the solution of the differential equation
$\frac{\d\theta}{\d t}(0) = \beta$ and $\theta(0) = {\widehat M}$. 
By unpacking the definitions, we obtain that, for every $\gamma \in H_1(S,\Sigma)$,
\begin{equation}
\label{eq:holtrem}
\begin{split}
\holx_{\theta(t)}(\gamma) 
& 
= \holx_{\widehat M}(\gamma) +t\beta(\gamma) \\
\holy_{\theta(t)}(\gamma)
& 
= \holy_{\widehat M}(\gamma).
\end{split}
\end{equation}
In \cite[Proposition 4.8]{2020tremors} it was shown that 
the solution to this differential 
equation has domain $\R$ for any non-atomic 
foliation cocycle $\beta$ and that the 
composition $\trem_\beta M \dfn \varpi(\trem_\beta {\widehat M})$ is well defined.
In particular, writing $\trem_{s\d y} M \dfn \trem_{s\holy}M$, we obtain that $u_s M = \trem_{s\holy_{M}}M$ by \Cref{eq:holtrem}. 

The following are properties of the tremor map and its interaction with other maps. See \cite[Section 6]{2020tremors} for more information.
\begin{PROP}
For every $M \in \H$, the following properties hold:
\begin{enumerate}
\item If $\beta_1,\beta_2 \in \tremspace_M$, then $\trem_{\beta_1 + \beta_2} M = \trem_{\beta_1}( \trem_{\beta_2}M)$, where $\beta_1 \in \tremspace_M$ can be identified with a unique foliation cocycle $\beta_1 \in \tremspace_{\trem_{\beta_2}M}.$
\item For any $s \in \R$ and $\beta \in \tremspace_M$,
\begin{equation*}
g_s (\trem_\beta M) = \trem_{e^{s}\beta} (g_sM) \qquad \text{and}\qquad u_s (\trem_{\beta} M)= \trem_{\beta} (u_sM).        
\end{equation*}
\item If $\Rel_t M$ is defined for $t \in \R$, then $\Rel_t \trem_\beta M = \trem_\beta \Rel_t M$.
\end{enumerate}
\end{PROP}

For the horizontally aperiodic, non-minimal surfaces in $\E$ considered in the main results, foliation cocycles admit a more concrete description in terms of the two torus components. This description will be used later in the proof of \Cref*{thm:mainthm}; see \Cref{eq:decomp_cocycle}.

\subsection{The locus \texorpdfstring{$\E$}{} and the slit construction}
\label{Section:Preliminaries}

We now specialize to the setting used in the main results. From this point on, our discussion focuses primarily on surfaces in $\H(1,1)$ which arise from gluing two identical tori along a slit.

\textcolor{black}{%
Let $T_1,T_2\in\mathcal H(0)$ be two flat tori. Let $I_i\subset T_i$ be an embedded segment starting at the marked point of $T_i$, for $i=1,2$. We orient each segment from the marked point to its other endpoint. Assume that $I_1$ and $I_2$ have the same length and direction.
The orientation of the surfaces and of the segments determines a left and a right side of each $I_i$. We form a new surface $M$ as a connected sum by cutting each torus along $I_i$ and identifying the left side of $I_1$ with the right side of $I_2$, and the right side of $I_1$ with the left side of $I_2$, by translations. The resulting topological surface has genus two. Since $I_1$ and $I_2$ have the same length and direction, we often suppress the subscripts and write simply
\[
M=T_1\#_I T_2.
\]
The image in $M$ of $I_1\cup I_2$ is called the \emph{slit}. The endpoints of $I_1$ and $I_2$ are identified in pairs, giving two distinguished points $\zeta_1,\zeta_2\in M$. We label them so that $\zeta_1$ is the point obtained from the marked points of $T_1$ and $T_2$, and $\zeta_2$ is the point obtained from the other endpoints of the segments.
The translation structures on $T_1$ and $T_2$ induce a translation atlas on $M$ away from $\zeta_1$ and $\zeta_2$. Away from the slit, we use the original charts on the tori. Around a point $x$ in the slit, such that $x \neq \zeta_1,\zeta_2$,
a neighborhood of $x$ is obtained by gluing two Euclidean half-disks of the same radius lying on opposite sides of $I_1$ and $I_2$. Since the corresponding charts on the two tori differ by a translation, and since the sides of the slits are identified by translations, these half-disk charts glue to give a well-defined translation chart near $x$.
Around each $\zeta_i$, the local model is obtained by taking small disks in $T_1$ and $T_2$ around the corresponding endpoints and gluing the opposite sides of the portions of $I_1$ and $I_2$ contained in these disks. This gives cone angle $4\pi$ at each $\zeta_i$ and 
\[
M\in\mathcal H(1,1).
\]
}
In \cite[Theorem 7.1]{McMullen2007DynamicsOS} it is proved that any surface $M \in \H(1,1)$ can be represented  by a slit construction $M=T_1 \# T_2$ with  tori $T_1,T_2 \in \H(0)$. 

A map $f:M \to N$ is a \emph{translation covering map} if
\[
f^{-1}(\Sigma_N)=\Sigma_M
\]
and
\[
f:M\setminus \Sigma_M \to N\setminus \Sigma_N
\]
is a covering map that preserves the translation structures. Let $\E\subset\H(1,1)$ be the collection of translation surfaces for which there exists a translation covering onto a surface in $\H(0,0)$.

Let $T_1$ and $T_2$ be identical copies of a torus $T\in\H(0)$, and let $I\subset T$ be a straight-line segment with one endpoint at the marked point. Denote by $p_1$ the marked point and by $p_2$ the other endpoint of $I$. We regard $T$, together with the distinguished points $p_1$ and $p_2$, as a translation surface in $\H(0,0)$.

The surface
\[
M=T\#_I T
\]
belongs to the locus $\E$. Indeed, let $\pi:M\to T$ be the map that sends $\zeta_i$ to $p_i$ and sends each point
\[
x\in M\setminus\{\zeta_1,\zeta_2\}
\]
to the corresponding point of $T$, after forgetting whether $x$ came from the first or second copy of $T$. Then $\pi$ is a translation covering map, $p_1$ and $p_2$ are the branch points, and $\zeta_1$ and $\zeta_2$ are the ramification points.

We have the following:

\begin{LEM}\label{lemma:Slits}
Let $M=T_1\#_{I}T_2 \in \E$ and $M'=T_1\#_{I'}T_2 \in \E$ be two translation surfaces with $T=T_1=T_2$ and the endpoints of $I$ and $I'$ be $p_1 \not=p_2\in T$. Then $M = M'$ if and only if $[I]=[I']$ in $H_1(T,\{p_1,p_2\};\Z/2\Z)$.
\end{LEM}
\textcolor{black}{%
\begin{proof}
See also \cite[Propositions~3.1 and~3.2]{2020tremors}.\\
By the classification of covering spaces, a connected $2$-fold covering of $T\setminus\{p_1,p_2\}$ branched over $p_1$ and $p_2$ corresponds uniquely to a surjective homomorphism
\[
\phi:H_1(T\setminus\{p_1,p_2\};\mathbb Z)\to \mathbb Z/2\mathbb Z
\]
with the property that
\[
\phi(\gamma_1)=\phi(\gamma_2)=1,
\]
where $\gamma_1$ and $\gamma_2$ are small loops winding once around $p_1$ and $p_2$, respectively. By the Universal Coefficient Theorem, this homomorphism is identified with a cohomology class
\[
\phi\in H^1(T\setminus\{p_1,p_2\};\mathbb Z/2\mathbb Z).
\]
Poincar\'e duality provides a natural isomorphism between the absolute cohomology of $T\setminus\{p_1,p_2\}$ and the homology of $T$ relative to $\{p_1,p_2\}$:
\[
\mathcal D:
H_1(T,\{p_1,p_2\};\mathbb Z/2\mathbb Z)
\xrightarrow{\cong}
H^1(T\setminus\{p_1,p_2\};\mathbb Z/2\mathbb Z).
\]
Under this duality, the isomorphism is given explicitly by the $\mathbb Z/2\mathbb Z$ geometric intersection pairing. The segments $I$ and $I'$ define relative homology classes
\[
[I],[I']\in H_1(T,\{p_1,p_2\};\mathbb Z/2\mathbb Z).
\]
Denote by
\[
\phi_I=\mathcal D([I])
\qquad\text{and}\qquad
\phi_{I'}=\mathcal D([I'])
\]
their dual cohomology classes. Since $I$ and $I'$ are segments with endpoints $p_1$ and $p_2$, we have
\[
\phi_I(\gamma_i)
=
[I]\cdot \gamma_i
=
[I']\cdot \gamma_i
=
\phi_{I'}(\gamma_i)
= 1 \pmod 2
\]
for $i=1,2$.\\
Finally, the translation surfaces $M$ and $M'$ are equal if and only if they define isomorphic coverings of the surface $T\in\H(0,0)$. This occurs if and only if the cohomology classes $\phi_I$ and $\phi_{I'}$ are equal. Since the map $\mathcal{D}$ is a bijection, $\phi_I=\phi_{I'}$ if and only if the underlying relative homology classes are equal:
\[
[I]=[I']
\quad\text{in}\quad
H_1(T,\{p_1,p_2\};\mathbb Z/2\mathbb Z).
\]
\end{proof}
}

This lemma will be fundamental for our computations. In many instances, we will represent the same surface differently, with a long slit, typically coming from the real Rel flow and a shorter slit in the same class of relative homology with coefficients in $\Z / 2\Z$. In other words, to check that the surfaces built by gluing with two slits $I$ and $I'$ (with endpoints $p_1$ and  $p_2$) are the same, it is enough to fix a basis of absolute homology, say $\delta, \gamma \in H_1(T)$ and compute the intersection numbers modulo 2. We use the notation $i(I,\delta)$, $i(I,\gamma)$, for this purpose. See for instance the proofs of \Cref{lemma:RelSubsequece} and \Cref{lemma:boundedarea}.

\subsubsection{On the horizontal flow and the slit construction}
This section introduces two lemmas that will be used frequently throughout this paper.
\begin{LEM}
$M \in \E$ is  horizontally periodic if and only if $M =T_1 \#_I T_2$, where $T$ is horizontally periodic.
\label{lemma:PeriodicSurf}
\end{LEM}

\begin{proof}
The proof follows from the definition of the locus $\E$. Let $\pi: M \to T$ be a translation cover of a torus $T$. A horizontal closed trajectory $\delta$ in $M$ maps to a closed trajectory $\pi(\delta),$ and this implies that $T$ is horizontally periodic. 
Conversely, any horizontal closed curve $\sigma$ in $T\backslash\{p_1,p_2\}$ could lift to a closed curve on $M$. If it does not lift to a closed curve, the pre-image $\pi^{-1}(\sigma)$ is a closed curve since it is the concatenation of the two possible lifts of $\sigma.$ Finally, any horizontal segment on $M\backslash \Sigma_M$ is contained in the pre-image of some horizontal trajectory in $T\backslash\Sigma_T$. According to the previous argument, it must be contained in a horizontal closed curve. This proves the other direction of the Lemma.
\end{proof}

\begin{LEM}
For $M \in \E$, the following are equivalent:
\begin{enumerate}
\item \label{part1} $M$ is horizontally aperiodic and non-minimal.
\item \label{part2} There exists a slit construction $M=T_1 \#_I T_2$, with $T$ horizontally minimal and $I$ a horizontal segment.
\end{enumerate}
\label{lemma:AperiodicNonMinimal}
\end{LEM} 

\begin{proof}
We only prove that \cref{part1} implies \cref{part2}. The other direction is simpler.
If $M$ is horizontally aperiodic, by \Cref{lemma:PeriodicSurf} $M=T_1\#_{I'}T_2$ with $T$ horizontally minimal. Suppose that $I'$
is not horizontal and consider the connected components of the horizontal flow. There are two minimal components by \cite[Theorem 1.1, item (2)]{Lindsey}.  
The boundary of each minimal component must contain a horizontal saddle connection; we call $I\subset T$ the image under $\pi$ of this saddle connection.
Since the closed curve $I'\cup I$ separates $M$ into two components, then $M= T_1\#_I T_2$.
\end{proof}

\subsection{Continued fractions}

For any $x \in \R$, denote by $\{x\}$ the \emph{fractional part} and by $\lfloor x \rfloor$  the \emph{integer part} of $x$.  
Let $\lVert x \rVert$ be $\min_{n \in \Z} |x - n | = \min\{\{x\},1-\{x\}\}$.
We say that $\alpha \in \R\backslash{\Q}$ is \emph{well approximable} if 
\begin{equation*}
\liminf_{q \in \N} q \lVert q \alpha \rVert = 0,
\end{equation*}
otherwise an irrational number $\alpha$ is \emph{badly approximable}.
Let $\alpha$ be an irrational number. The \emph{continued fraction expansion} of $\alpha$ is the expression
\begin{equation*}
\alpha =a_0 + \frac{1}{a_1
+\frac{1}{a_2   + \frac{1}{a_3 + \frac{1}{\ddots}} 
}
}.
\end{equation*}
The \emph{best approximants} of $\alpha$ are the rational numbers
\begin{equation*}
\frac{p_n}{q_n} =a_0 + \frac{1}{a_1
+\frac{1}{a_2   + \frac{1}{a_3 + \frac{1}{\ddots +\frac{1}{a_n}}} 
}
}.
\end{equation*}
The integers $a_0\geq0$, $a_1,a_2,\dots \geq 1$ are called \emph{partial quotients.}

\begin{THM}    
For $\alpha \in \R\backslash\Q,$ the following holds:
\begin{enumerate}[label=(\arabic*)]
\label{BigThmContFrac}

\item 
\label{item:defofpq}
\cite[Theorem 1]{Khinchin} For every $k \geq 2$,
\begin{equation*}
\begin{split}
p_k &= a_kp_{k-1}+p_{k-2},\\
q_k &= a_kq_{k-1}+q_{k-2},
\end{split}
\end{equation*}
and $\gcd(p_k,q_k)=1.$
\item
\label{KhinchinInequality}
\cite[Theorems 9 and 13]{Khinchin} For $k\geq 1$, then
\begin{equation*}
\frac{1}{q_{k+1}+q_k} < \lVert q_k \alpha \rVert < \frac{1}{q_{k+1}} < \frac{1}{a_{k+1}q_{k}}.
\end{equation*}
\item
\label{item:bestapproximation}
\cite[Theorem 16]{Khinchin} For any $m \in \{1,2,\dots,q_{k+1}-1\}$, $m\not=q_k$ and $k \geq 1$, then
\begin{equation*}
\lVert q_k \alpha \rVert < \lVert m\alpha \rVert. 
\end{equation*}
\item \label{item:equivalencesofbaddlyaprrox} The following are equivalent: 
\begin{enumerate}
\item $\exists c>0$ such that $\forall k \geq 1$, then $q_k\lVert q_k \alpha\rVert > c$, \label{baddlyapproxa}

\item $\exists K> 0$ such that  $\forall i \geq 1$, then $a_i \leq K$,
\label{baddlyapproxb}

\item $\exists c'>0$ such that $\forall k \geq 1$,  $\frac{q_k}{q_{k+1}} \geq c',$
\label{baddlyapproxc}

\item $\alpha$ is badly approximable.
\label{baddlyapproxd}
\end{enumerate}

\end{enumerate}
\end{THM}

\begin{proof}
We will prove \Cref{item:equivalencesofbaddlyaprrox} for completeness by proving the implications
\Cref{baddlyapproxa} $\Rightarrow$ \Cref{baddlyapproxc} $\Rightarrow$ \Cref{baddlyapproxb} $\Rightarrow$ \Cref{baddlyapproxa}
and \Cref{baddlyapproxa} $\Leftrightarrow$ \Cref{baddlyapproxd}. 

If \Cref{baddlyapproxa} holds, then by \Cref{KhinchinInequality} we obtain that 
\begin{equation*}
c<q_k\lVert q_k \alpha \rVert < \frac{q_k}{q_{k+1}},
\end{equation*}
which proves \Cref{baddlyapproxc}.

If \Cref{baddlyapproxc} is true, then let $K $ be equal to $\frac{1}{c'}$ and observe that by \Cref{item:defofpq}
\begin{equation*}
a_{k+1} < a_{k+1} +\frac{q_{k-1}}{q_k} =\frac{q_{k+1}}{q_k} \leq \frac{1}{c'}=K,
\end{equation*}
which implies \Cref{baddlyapproxb}. 

Now suppose that \Cref{baddlyapproxb} is true. Then by \Cref{KhinchinInequality} we obtain
\begin{equation*}
c \dfn \frac{1}{K+2}< \frac{1}{a_{k+1} + \frac{q_{k-1}}{q_k} + 1}=\frac{1}{\frac{q_{k+1}}{q_k} + 1} < q_k \lVert q_k \alpha \rVert
\end{equation*}
which proves \Cref{baddlyapproxa}.

Finally, we prove that \Cref{baddlyapproxa}  is equivalent to \Cref{baddlyapproxd}. Indeed, for every positive integer $m$, let $k\geq 1$ such that $q_k \le m <q_{k+1}$. Then using \Cref{item:bestapproximation} we see that 
\begin{equation*}
c < q_{k}\lVert q_k \alpha \rVert \le m \lVert m \alpha \rVert.
\end{equation*}
This proves the equivalence since the previous inequality implies that
\begin{equation*}
0 < c \le \liminf_{k \geq 0} q_{k}\lVert q_k \alpha \rVert = \liminf_{m \geq 1} m \lVert m\alpha \rVert.
\end{equation*}
\end{proof}

\section{Real Rel and area exchange}
\label{sec:RelRec}

In this section, we introduce the area exchange associated to two slit presentations of the same surface in $\E$. The purpose of this definition is to compare period-coordinate charts that arise from different slit representatives. Indeed, if
\[
M=T_1\#_I T_2,
\]
then applying real Rel changes only the relative period of the slit. Thus, for large values of the real Rel parameter, the slit in the presentation of $\Rel_s M$ has large horizontal holonomy. This gives a period-coordinate chart, but it is often more convenient to represent the same surface using another slit with shorter holonomy. The area exchange measures how the passage from one slit presentation to another modifies the decomposition of the surface into the two torus components. Later, this will allow us to track how the corresponding change of period coordinates affects the foliation cocycles.

Suppose that
\[
T_1\#_{I}T_2
\qquad\text{and}\qquad
T_1\#_{I'}T_2
\]
represent the same translation surface in $\E$. The slits $I$ and $I'$ are straight segments with the same endpoints. We denote by $I\cup I'$ the closed curve determined by the union of these two segments. The class of this closed curve is homologous to zero in
\[
H_1(T;\Z/2\Z).
\]
The curve $I\cup I'$ separates $T$ into parallelograms that can be colored with two colors so that no two adjacent parallelograms have the same color. This colored decomposition records how to shuffle the parallelograms in the presentation $T_1\#_I T_2$ and glue them back to obtain the presentation $T_1\#_{I'}T_2$; see \cite[Lemma 10.5]{2020tremors}. 

\begin{figure}
\centering
\includegraphics[trim={7cm 8cm 4cm 9cm},scale=0.5]{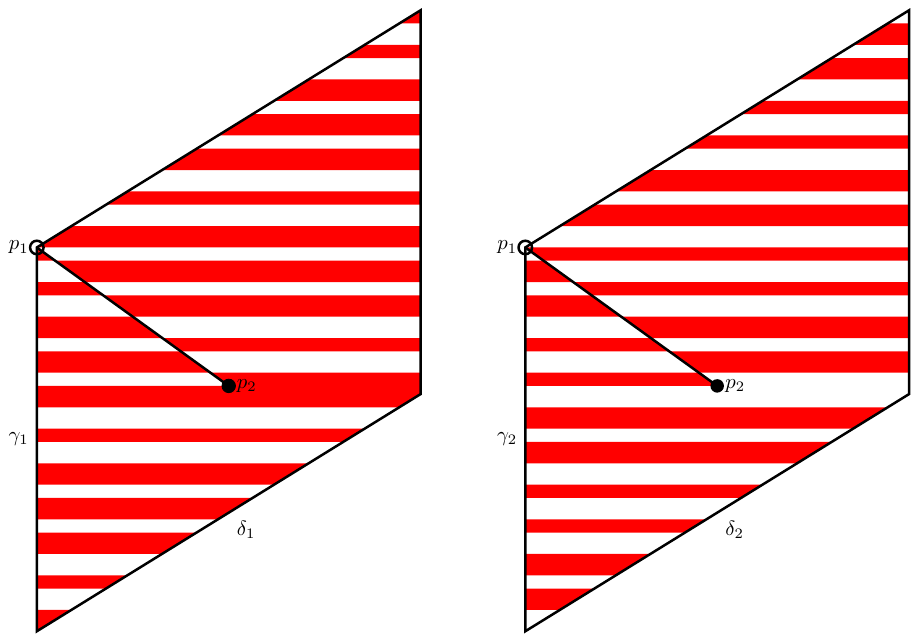}
\caption{The surface $M$ has a slit between the singularities $p_1$ and $p_2$. The regions of different colors are minimal components of the horizontal flow.}
\label{fig:AreaExchange}
\end{figure}
\textcolor{black}{
Denote by $A_1$ and $A_2$ the two colored regions in $T$ separated by the curve $I \cup I'$.
}
The \emph{area exchange} is the pair of numbers $(B_1,B_2)$ where $B_i = {\rm area}(A_i)$ for $i=1,2$ and the \emph{area imbalance} is  $\frac{|B_1-B_2|}{{\rm area}(T)}$. In \Cref{fig:AreaExchange}, the area exchange is the area of each colored region restricted to one of the copies of the torus $T$.
In \cite[Section 10.1]{2020tremors}, the union $I\cup I'$ is called  \emph{checkerboard}.

We will make some simplifying assumptions by taking advantage of the continuity of the $\operatorname{SL}(2,\R)$ action and the fact that $\E$ is $u_s$ and $g_t$ invariant.

We will suppose that $M \in \E$ is vertically periodic. Then we can write $M = T_1 \#_I T_2$ where $T$ is the quotient $\R^2/ \Lambda$ by a lattice $\Lambda = (0,b)\Z\otimes (\frac{a}{b},\alpha b)\Z$, with $a,\alpha, b > 0 $. 

By replacing $M$ with $g_{\ln(b)}M$, we can assume $b=1$. Then the surface is  vertically periodic and the length of a vertical closed curve is 1.
On $T$, the first return time of the horizontal flow to a vertical closed curve of length $1$ is conjugate to the map $R_\alpha:[0,1) \to [0,1)$, $x \mapsto x+\alpha \bmod{1}$.

\begin{LEM}
\label{lemma:short_slit_representative}
Let
\[
T=\mathbb R^2/\Lambda,
\qquad
\Lambda=\langle (0,1),(a,\alpha)\rangle,
\]
and let $P:\mathbb R^2\to T$ be the universal covering map, given by
\[
P\bigl(s(0,1)+t(a,\alpha)\bigr)
=
(s \bmod \mathbb Z)(0,1)+(t \bmod \mathbb Z)(a,\alpha).
\]
Suppose that $M$ can be represented as $T_1\#_I T_2$, where 
\[
I\subset T=T_1=T_2
\]
is a horizontal segment. Then there exists a segment $I'\subset T$ such that
\[
[I]=[I']
\quad\text{in}\quad
H_1(T,\{p_1,p_2\};\mathbb Z/2\mathbb Z),
\]
and such that $I'$ is the image under $P$ of a straight segment contained in a fundamental parallelogram for $T$. In particular, the length of $I'$ is bounded above by the diameter of $T$.
\end{LEM}

The choice of the correct segment $I'$ can be read from the intersection computations in \Cref{fig:IntersectionNumbers}: after fixing generators of relative homology, one compares the mod $2$ intersection numbers with these generators.

\begin{proof}
Let $p_1$ and $p_2$ be the endpoints of $I$, and choose coordinates so that $p_1=P(0,0)$. Let
\[
\widetilde p_2=(xa,x\alpha+y)
\]
be the lift of $p_2$ contained in the fundamental parallelogram
\[
\{s(0,1)+t(a,\alpha):0\leq s,t<1\},
\]
where $0\leq x,y<1$.

Consider the four vertices
\[
(0,0),\qquad (0,1),\qquad (a,\alpha),\qquad (a,\alpha+1)
\]
of this fundamental parallelogram. By \Cref{lemma:Slits}, one of the four straight segments joining $\widetilde p_2$ to one of these vertices projects under $P$ to a segment $I'\subset T$ satisfying
\[
[I]=[I']
\quad\text{in}\quad
H_1(T,\{p_1,p_2\};\mathbb Z/2\mathbb Z).
\]
The segment $I'$ is contained in the image of the fundamental parallelogram, so its length is bounded by the diameter of $T$.
\end{proof}

\textcolor{black}{We now prove recurrence for tremors of surfaces in the locus $\E$ when the slope parameter $\alpha$ is well approximable.}

\begin{THM}
\label{prop:recurrencecase}
Suppose that $\alpha$ is well approximable. Let $\beta=\beta_v$, and assume that
$v\neq c\,\d y$
for every $c\in\mathbb R.$
Then there exists an increasing sequence $\{t_j\}$ with $t_j\to\infty$ such that
\[
\Rel_{t_j}\trem_{\beta}M \to \trem_{\beta}M.
\]
\end{THM}
\textcolor{black}{%
We first prove a normalized version of recurrence. In the following \namecref{lemma:RelSubsequece}, the slit is assumed to be horizontal of length $a$. This is a special case of the preceding \namecref{prop:recurrencecase}.
}

\begin{LEM}
\label{lemma:RelSubsequece}
Suppose that
\[
M=T_1\#_I T_2,
\qquad
T=\mathbb R^2/\bigl((0,1)\mathbb Z\oplus (a,\alpha)\mathbb Z\bigr),
\]
and that $I$ is a horizontal segment of length $a$. Then there exists a sequence of surfaces
\[
M_j \dfn T_1\#_{I_j}T_2
\]
such that $I_j$ is a horizontal segment of length $(2q_j+1)a$ and
\[
M_j\to M.
\]
Moreover, for each $j$, there is a shorter segment $I'_j$ with the same endpoints as $I_j$ such that $I'_j\to I$, and, for all sufficiently large $j$, the area imbalance between $I_j$ and $I'_j$ is
\[
1-(2q_j+1)\|q_j\alpha\|.
\]
The surfaces $M_j$ are different from $M$, and
\[
M_j=\Rel_{t_j}M
\qquad\text{where}\qquad
t_j=2q_ja.
\]
In particular,
\[
\Rel_{t_j}M\to M
\]
as $j\to\infty$, so the positive real Rel trajectory of $M$ is recurrent.
\end{LEM}

\begin{proof}
Let $q=q_j$ be a denominator of a best approximation of $\alpha$ such that $\|q\alpha\|<1/2$, and let
\[
P:\mathbb R^2\to T
\]
be the universal covering map. Suppose that the left endpoint of $I$ lifts to $(0,0)\in\mathbb R^2$. To simplify notation, write
\[
I_q=I_j
=
P\bigl(\{(t,0):0\leq t\leq (2q+1)a\}\bigr).
\]
The segment $I_q$ intersects the vertical segment
\[
\delta=P\bigl(\{(0,s):0\leq s\leq 1\}\bigr)
\]
at the points
\[
P(0,0)=P(0,1),\,
P(0,1-R_\alpha(0)),\dots,\,
P(0,1-R_\alpha^{2q+1}(0)).
\]

The shorter segment with the same endpoints is
\[
I'_q=I'_j
=
P\left(
\left\{
\left(as,1+s\bigl(\alpha-R_\alpha^{2q+1}(0)\bigr)\right):
0\leq s\leq 1
\right\}
\right).
\]
The segments $I_q$ and $I'_q$ intersect three times: at their two endpoints and at the midpoint. Indeed, this midpoint is
\[
P\left(\frac a2,1+\frac{\alpha-R_\alpha^{2q+1}(0)}2\right)
=
P\left(\frac a2,1-\frac{R_\alpha^{2q}(0)}2\right).
\]
\textcolor{black}{%
We now explain why this is the only intersection point in the interior. By \Cref{BigThmContFrac}, \cref{item:bestapproximation}, if
\[
1\leq m\leq q_{j+1}-1
\qquad\text{and}\qquad
m\neq q=q_j,
\]
then
\[
\|q\alpha\|<\|m\alpha\|.
\] 
After ordering the points $R_\alpha^k(0)$, with $0\leq k\leq 2q$, in the interval $[0,1)$, either
\[
0<R_\alpha^q(0)<R_\alpha^{2q}(0)
\]
or
\[
R_\alpha^{2q}(0)<R_\alpha^q(0)<1.
\]
Moreover, there is no other point $R_\alpha^k(0)$ between $0$ and $R_\alpha^q(0)$, nor between $R_\alpha^q(0)$ and $R_\alpha^{2q}(0)$. Indeed, suppose first that $0<k<q$. If $R_\alpha^k(0)$ lies between $0$ and $R_\alpha^q(0)$, then $\|k\alpha\|<\|q\alpha\|$, contradicting \Cref{BigThmContFrac}, \cref{item:bestapproximation}. If $R_\alpha^k(0)$ lies between $R_\alpha^q(0)$ and $R_\alpha^{2q}(0)$, then
\[
\|R_\alpha^k(0)-R_\alpha^q(0)\|=\|(q-k)\alpha\|<\|q\alpha\|.
\]
Letting $m=q-k$, we have $0<m<q$, again contradicting \Cref{BigThmContFrac}, \cref{item:bestapproximation}.\\
Now suppose that $q<k<2q$. If $R_\alpha^k(0)$ lies between $0$ and $R_\alpha^q(0)$, or between $R_\alpha^q(0)$ and $R_\alpha^{2q}(0)$, then
\[
\|R_\alpha^k(0)-R_\alpha^q(0)\|=\|(k-q)\alpha\|<\|q\alpha\|.
\]
Letting $m=k-q$, we have $0<m<q$, which contradicts \Cref{BigThmContFrac}, \cref{item:bestapproximation}. 
This proves that there is exactly one intersection in the interior between $I_q$ and $I'_q$.\\
The vertical length of $I'_q$ is $\bigl|\alpha-R_\alpha^{2q+1}(0)\bigr|=  \|2q\alpha\|$. The intersection happens when the vertical distance is $\|q\alpha\|$. Hence the intersection occurs halfway along $I'_q$,  and therefore is at the midpoint of both $I_q$ and $I_q'$.
}
\begin{CLAIME} 
We have
\[
T_1\#_{I_q}T_2 = T_1\#_{I'_q}T_2.
\]
\end{CLAIME}
\begin{proof}
\renewcommand{\qedsymbol}{$\triangle$}
Let us fix representatives of generators of $H_1(T)$ in order to compute the intersection numbers with $I_q$ and $I'_q$. This will allow us to apply \Cref{lemma:Slits} and prove the claim.
As before, let $\delta\subset T$ be the closed curve given by the $P$-image of the segment joining $(0,0)$ and $(0,1)$ in $\mathbb R^2$. Let $\gamma\subset T$ be the closed curve given by the $P$-image of the segment joining $(0,0)$ and $(a,\alpha)$.

The intersection number $i(I_q,\delta)$ is $2(q+1)$, as seen above. Also, $i(I'_q,\delta)=2$, since the intersection points are $P(0,0)$ and $P(0,1-R_\alpha^{2q+1}(0))$.

Denote by $p=p_j$ the numerator of the best approximant $p_j/q_j$. Then $i(I_q,\gamma)$ is the number of intersections of the curve $\{(t,0):0\leq t\leq (2q+1)a\}$ with the diagonal lines $y=\frac{\alpha}{a}x-n$, for $n\in\mathbb N$. This number is equal to $m+1$, where $m$ is the largest positive integer such that the lift of $I_q$ starting at $0\in\mathbb R^2$, namely $\{(t,0):0\leq t\leq (2q+1)a\}$, intersects all the lines $y=\frac{\alpha}{a}x-n$ for $n=0,\dots,m$.

To determine $m$, notice that the endpoint $((2q+1)a,0)$ lies below the line $y=\frac{\alpha}{a}x-m$ and above the line $y=\frac{\alpha}{a}x-(m+1)$. This is equivalent to
\[
(2q+1)\alpha-(m+1)<0<(2q+1)\alpha-m,
\]
which is equivalent to
\[
m<(2q+1)\alpha<m+1.
\]
Moreover, $m=2p_j$ for all sufficiently large $j$. Indeed, \Cref{BigThmContFrac}, \Cref{KhinchinInequality} implies that
\[
-\frac{1}{q_{j+1}}<q_j\alpha-p_j<\frac{1}{q_{j+1}}.
\]
This implies that, for all sufficiently large $j$,
\[
0<(2q_j+1)\alpha-2p_j<\frac{2}{q_{j+1}}+\alpha<1.
\]
Thus $i(I_q,\gamma)=2p_j+1$.

Moreover, $i(I'_q,\gamma)=1$, because any lift of $I'_q$ to $\mathbb R^2$ intersects the diagonal lines $y=\frac{\alpha}{a}x-m$ only once, for all $q=q_j$ sufficiently large. For example, the lift starting at $(0,0)$ intersects $y=\frac{\alpha}{a}x$ only at $(0,0)$.

Strictly speaking, the intersection pairing used in \Cref{lemma:Slits} is with closed curves in $T\setminus\{p_1,p_2\}$. The curves $\delta$ and $\gamma$ pass through $p_1$ and $p_2$, so we replace them by small perturbations $\delta'$ and $\gamma'$ in $T\setminus\{p_1,p_2\}$ representing the corresponding homology classes. 
Therefore $i(I_q,\delta') = 2q+1$, $i(I'_q,\delta')=1$, $i(I_q,\gamma') = 2p_j$, and $i(I'_q,\gamma')=0$.
Hence
\[
[I_q]=[I'_q]
\quad\text{in}\quad
H_1(T,\{p_1,p_2\};\mathbb Z/2\mathbb Z).
\]
\end{proof}
To compute the area exchange, notice that $I_q$ and $I'_q$ separate $T$ into two trapezoids, since these segments share endpoints and intersect in their interiors exactly once. Both trapezoids have base of length $(q+\frac{1}{2})a$, and one of them has height $\lVert q\alpha\rVert$. Since the area of $T$ is $a$, the areas of the two trapezoids are
\[
S=\left(q+\frac{1}{2}\right)a\lVert q\alpha\rVert
\qquad\text{and}\qquad
a-S.
\]
Thus the area imbalance is
\[
\frac{|a-2S|}{a}
=
1-(2q+1)\lVert q\alpha\rVert.
\]
The rest of \Cref{lemma:RelSubsequece} follows from the fact that $I'_q\to I$ as $q\to\infty$. Therefore,
\[
T_1\#_{I'_q}T_2 \to T_1\#_I T_2.
\]
This completes the proof.
\end{proof}

The preceding lemma provides the main mechanism for the proof of \Cref{prop:recurrencecase}. This theorem proves one direction of \Cref{thm:mainthm} by showing that the recurrence behavior of the real Rel trajectory persists under the tremors considered.
The idea is to use recurrence times $t_j\to\infty$ for which $\Rel_{t_j}M$ is close to $M$. At these times, $\Rel_{t_j}M$ admits a slit presentation whose short slit is close to the original slit of $M$. The area exchange between the two slit presentations controls the corresponding change of period coordinates after applying the tremor, and this shows that $\Rel_{t_j}\trem_\beta M$ is close to $\trem_\beta M$.

\begin{proof}[Proof of \Cref{prop:recurrencecase}]
Recall that we are assuming that $T=T_1=T_2$, where
\[
T=\mathbb R^2/\bigl((0,1)\mathbb Z\oplus (a,\alpha)\mathbb Z\bigr).
\]
We can assume that $M'$ is of the form
\[
M'=T_1\#_{I_\ell}T_2,
\]
where $I_\ell$ is a horizontal segment with holonomy $(\ell,0)$ and $\ell>0$ is arbitrary. 
\textcolor{black}{
Let
\[
M=T_1\#_I T_2,
\]
where $I$ is a horizontal segment with holonomy $(a,0)$. By \Cref{item:RelCondition}, the surface $\Rel_s M$ is defined for all $s>-a$. Taking $s=\ell-a$, we obtain
\[
M'=\Rel_{\ell-a}M.
\]
}

\begin{CLAIME}
With $M$ and $M'$ as above, it is enough to prove \Cref{prop:recurrencecase} in the case where
\[
M=T_1\#_I T_2
\]
and $I$ is a horizontal segment of length $a$. More precisely, suppose that there exists a sequence $t_j\to\infty$ such that
\[
\Rel_{t_j}\trem_\beta M \to \trem_\beta M.
\]
Then
\[
\Rel_{t_j}\trem_\beta M' \to \trem_\beta M'.
\]
\end{CLAIME}

\begin{proof} 
\renewcommand{\qedsymbol}{$\triangle$}
By \cite[Proposition 6.7]{2020tremors},
we have the equation:
\begin{equation} 
\Rel_s \Rel_{t_j} \trem_{\beta}M = 
\Rel_{t_j} \trem_{\beta}\Rel_s M =
\Rel_{t_j} \trem_{\beta} M'.
\label{eq:groupLaw}
\end{equation}
Assume $\Rel_s\trem_\beta M$ is defined with $s=\ell-a>0$. We want to show that the map $\Rel_s(\cdot)$ is well defined and continuous in a neighborhood of $\trem_\beta M$. We will use  \cite[Proposition 4.3]{HoroDynamics} and their notation regarding the Rel map (not just real Rel). By \cite[Proposition 4.3]{HoroDynamics}, the set $\Omega \subset \H \times \mathfrak{R}$ of pairs $(N,v)$ for which $\Rel_vN$ is defined, is open and the map $(N,v) \mapsto \Rel_vN$ is continuous on $\Omega$. The projection maps $\pi_1: \Omega \to \H$ and  $\pi_2:\Omega \to \mathfrak{R}$  are continuous and open. Also, the projection map $P_x:\mathfrak{R}\to\mathfrak{R}_x$, $(x,y)\mapsto x$ is continuous.
Since $(\trem_\beta M,(s,0))$ is in $\Omega$, there exists  $\varepsilon>0$ such that $O\dfn \pi_1\left(\pi_2^{-1}(P_x^{-1}(s-\varepsilon,s+\varepsilon))\right)$ is an open subset of $\Omega$ and contains $\trem_\beta M$. Moreover, the map $N \mapsto  \Rel_sN$ is defined and continuous on $O$, since it is the map $(N,v)\mapsto \Rel_vN$ restricted to $\{(N,(s,0)): N \in O\}\subset \Omega$.

Thus, combining \Cref{eq:groupLaw} and continuity of the map $\Rel_s(\cdot)$ at $\trem_\beta M$ allows us to take the limit:  
\begin{equation*}
\lim_{j \to \infty}\Rel_{t_j} \trem_{\beta} M'
= \lim_{j \to \infty}\Rel_s \Rel_{t_j} \trem_{\beta}M = \Rel_s \trem_\beta M = \trem_\beta M'. 
\end{equation*}
\end{proof}
\textcolor{black}{By \Cref{lemma:RelSubsequece}, if we set $t_j=2q_j a$, then
\[
\Rel_{t_j}M\to M
\]
as $j\to\infty$.} We now show that, under the well-approximability assumption on $\alpha$ in \Cref{prop:recurrencecase}, the same sequence gives recurrence for the tremor of $M$; that is,
\[
\Rel_{t_j}\trem_\beta M \to \trem_\beta M.
\]

We work in a fixed period coordinate chart near $M$ adapted to the slit construction. In this chart, a surface close to $M$ is represented by a tuple
\[
(\gamma_1,\delta_1,\gamma_2,\delta_2,J)\in(\mathbb R^2)^5,
\]
where $\gamma_i$ and $\delta_i$ are the periods of the two generators of the torus $T_i$, and $J$ is the holonomy vector of the slit. Thus $T_i=\mathbb R^2/\Lambda_i$, where
\[
\Lambda_i=\gamma_i\mathbb Z\oplus \delta_i\mathbb Z.
\]
For the surface
\[
M=T_1\#_I T_2
\]
under consideration, we have
\[
\gamma_1=\gamma_2=(0,1),
\qquad
\delta_1=\delta_2=(a,\alpha),
\qquad
J=I=(a,0).
\]
Hence $M$ is represented in these coordinates by
\[
((0,1),(a,\alpha),(0,1),(a,\alpha),(a,0)).
\]
Let
\[
M_j=\Rel_{2q_ja}M.
\]
By \Cref{lemma:RelSubsequece}, the surface $M_j$ can be written as
\[
M_j=T_1\#_{I'_j}T_2,
\]
where $I'_j$ has the same endpoints as the horizontal slit $I_j$ and its holonomy vector converges to the holonomy of $I$. More precisely, in the notation of the proof of \Cref{lemma:RelSubsequece}, the holonomy vector of $I'_j$ is
\[
J_j=(a,(-1)^{j+1}2\|q_j\alpha\|).
\]
Therefore, in the same period-coordinate chart, $M_j$ is represented by
\[
((0,1),(a,\alpha),(0,1),(a,\alpha),J_j).
\]
Thus the torus periods remain fixed along the sequence $M_j$, and the only coordinate that changes is the slit holonomy $J_j$, which converges to $J=(a,0)$.

Since $M$ is aperiodic and not minimal, there are two minimal components; see, for instance, \cite[Theorem 1.1 item (2)]{Lindsey}. 
\textcolor{black}{%
These components are tori where the slit is a boundary, and minimality of the horizontal flow on each component implies unique ergodicity on that component. In particular, the restriction of Lebesgue measure to each component is an ergodic invariant measure. Thus the transverse measure defining the tremor can be written as a linear combination of the two invariant transverse measures supported on the two tori. Equivalently, the foliation cocycle decomposes as
\begin{equation}\label{eq:decomp_cocycle}
\beta=s_1\beta_1+s_2\beta_2,
\end{equation}
where $\beta_i$ is the foliation cocycle induced by the restriction of $\d y$ to the $i$th minimal component.
}
Equivalently, there are real numbers $s_1$ and $s_2$ such that, in the period coordinate chart adapted to the slit construction, the tremor $\trem_\beta M$ is represented by
\[
\left(
u_{s_1}\gamma_1,
u_{s_1}\delta_1,
u_{s_2}\gamma_2,
u_{s_2}\delta_2,
I
\right).
\]
Here the first two coordinates are the periods of the first torus, the next two are the periods of the second torus. The last coordinate is the slit holonomy $(a,0)$ since it is horizontal.

\textcolor{black}{%
We fix a period-coordinate chart adapted to the slit construction and containing $\trem_\beta M$. It is enough to show that, after passing to a subsequence, the period coordinates of $\Rel_{t_j}\trem_\beta M$ in this chart converge to the period coordinates of $\trem_\beta M$.}

Since $\alpha$ is well approximable, there exists an increasing sequence $(j_k)_{k\in\mathbb N}$ such that
\[
a_{j_k+1}\to\infty .
\]
By \Cref{BigThmContFrac},
\[
q_{j_k}\lVert q_{j_k}\alpha\rVert
<
\frac{1}{a_{j_k+1}}
\to 0
\qquad
\text{as}
\qquad
k\to\infty .
\]
Let
\[
t_{j_k}=2q_{j_k}a .
\]
\textcolor{black}{
For the sequence of slits $I_{j_k}$ and $I'_{j_k}$ from \Cref{lemma:RelSubsequece}, define
\[
\kappa_{j_k}
:=
\left(q_{j_k}+\frac12\right)
\lVert q_{j_k}\alpha\rVert a .
\]
Then $\kappa_{j_k}\to 0$. The area imbalance between $I_{j_k}$ and $I'_{j_k}$ is
\[
1-\frac{2\kappa_{j_k}}{a}
=
1-(2q_{j_k}+1)\lVert q_{j_k}\alpha\rVert,
\]
and hence tends to $1$. Equivalently, the smaller exchanged area tends to $0$.
}

To avoid excessive notation, from now until the end of the proof we write $j$ for $j_k$. Thus $t_j$ means $t_{j_k}$ and $\kappa_j$ means $\kappa_{j_k}$.

Recall that
\[
\trem_\beta \Rel_{t_j}M
=
\Rel_{t_j}\trem_\beta M.
\]
\textcolor{black}{Applying \Cref{eq:holtrem} to the marked surface $\Rel_{t_j}M$ gives
\[
\hol_{\Rel_{t_j}\trem_\beta M}
=
\hol_{\Rel_{t_j}M}+(\beta,0).
\]}
\textcolor{black}{%
The key point is that, along the subsequence chosen above, the smaller coordinate of the area exchange tends to zero. Consequently, after passing from the long-slit presentation to the short-slit presentation, each period coordinate of $\Rel_{t_j}\trem_\beta M$ is a linear combination of two torus periods: the corresponding period of $\trem_\beta M$ with coefficient tending to $1$, and a period from the other torus with coefficient tending to $0$. Hence these period coordinates converge to the corresponding period coordinates of $\trem_\beta M$.}

Using the decomposition of $\beta$ from \Cref{eq:decomp_cocycle}, and using the area imbalance between $I_j$ and $I'_j$, the first coordinate of $\Rel_{t_j}\trem_\beta M$ in the fixed period-coordinate chart is
\[
(1-O(\kappa_j))u_{s_1}\gamma_1
+
O(\kappa_j)u_{s_2}\gamma_2,
\]
and the second coordinate is
\[
(1-O(\kappa_j))u_{s_1}\delta_1
+
O(\kappa_j)u_{s_2}\delta_2.
\]
Similarly, the third and fourth coordinates are
\[
(1-O(\kappa_j))u_{s_2}\gamma_2
+
O(\kappa_j)u_{s_1}\gamma_1
\]
and
\[
(1-O(\kappa_j))u_{s_2}\delta_2
+
O(\kappa_j)u_{s_1}\delta_1.
\]

It remains to check the slit coordinate. The fifth coordinate of $\Rel_{t_j}M$ is the holonomy of $I'_j$, and by \Cref{lemma:RelSubsequece} we have
\[
I'_j\to I.
\]
Moreover, since the smaller exchanged area is $O(\kappa_j)$, the contribution of the cocycle $\beta$ to the slit coordinate satisfies
\[
\beta(I'_j)=(a+O(\kappa_j),O(\kappa_j)).
\]
The slit coordinate of $\Rel_{t_j}\trem_\beta M$ is
\[
(a+O(\kappa_j),O(\kappa_j)),
\]
because $I'_j\to I$ and the contribution of the smaller colored region has area $O(\kappa_j)$.
Thus the slit coordinate of $\Rel_{t_j}\trem_\beta M$ converges to the slit coordinate of $\trem_\beta M$.

Therefore all five period coordinates of $\Rel_{t_j}\trem_\beta M$ converge to the corresponding period coordinates of $\trem_\beta M$. Hence
\[
\Rel_{t_j}\trem_\beta M
\to
\trem_\beta M.
\]
\end{proof}

\section{Non-recurrence of real Rel}

\label{sec:proofoflemmaboundedarea}

We will prove the following theorem at the end of this section.

\begin{THM}
Let $M \in\E$ be a horizontally aperiodic, non-minimal surface and let $\alpha$ be badly approximable. Then for
\[
\beta=\beta_v
\quad
\text{with}
\quad
v \not \in \{ c \d y: c \in \R\backslash\{0\}\},
\]
and for every sequence of positive numbers $t_i \to \infty$,
\begin{equation*}
\trem_\beta M \not \in \bar{\{\Rel_{t_i}\trem_\beta M\}}.
\end{equation*}
\label{thm:NonRecurrenceofRel}
\end{THM}

The first step towards proving \Cref{thm:NonRecurrenceofRel} is to prove a \namecref{lemma:boundedarea} about the area exchange of two slits.

\textcolor{black}{%
We use the normalization introduced at the beginning of \Cref{sec:RelRec}. Namely, after applying the continuous actions $g_t$ and $u_r$, we may assume that the torus is
\begin{equation*}
T=\mathbb R^2/\Lambda,
\qquad
\Lambda=(a,\alpha)\mathbb Z\oplus (0,1)\mathbb Z,
\end{equation*}
where $a>0$ and $\alpha\in(0,1)$ is irrational. This normalization is only a simplification of coordinates and does not change the recurrence or non-recurrence question under consideration.
}

The surface $\Rel_s M$ can be represented as a slit construction with a long horizontal slit $I$, but it can also be represented with a different, non-horizontal slit $I'$. More precisely, we consider pairs of slits $I$ and $I'$ in $T$ with the same endpoints such that
\[
\Rel_sM = T_1 \#_I T_2 = T_1 \#_{I'} T_2,
\]
where $I$ is horizontal and $I'\neq I$. \textcolor{black}{%
Recall that by \Cref{lemma:Slits}, the slits $I$ and $I'$ determine the same relative homology class modulo $2$.
}

We would like to determine how close $\trem_\beta M$ and $\Rel_s\trem_\beta M$ can be. When $\alpha$ is badly approximable, the next lemma shows that the area exchange between $I$ and any such alternative slit $I'$ is bounded below by a positive constant independent of the length of the horizontal slit.

\begin{LEM}
Suppose that $I$ is a horizontal segment in $T$ with $\length(I)>a$, and suppose that $I'$ is a segment in $T$ such that $I'\neq I$ and
\[
T_1 \#_{I} T_2 = T_1 \#_{I'} T_2.
\]
If $\alpha$ is badly approximable, then the pair of numbers in the definition of area exchange between $I$ and $I'$ are bounded below by a positive constant independent of the lengths of $I$ and $I'$.
\label{lemma:boundedarea}
\end{LEM}

\begin{proof}
Let $\zeta_1,\, \zeta_2$ be the left and right endpoints of $I$, respectively, and let $P:\mathbb R^2\to T$ be the universal cover. 

In $\mathbb R^2$, consider the vertical segment with endpoints $(0,0)$ and $(0,1)$, and the diagonal segment with endpoints $(0,0)$ and $(a,\alpha)$. These segments project to closed curves in $T$. We denote their projections by $\gamma$ and $\delta$, respectively. These curves generate $H_1(T)$.

\textcolor{black}{%
In order to apply the intersection pairing with  \Cref{lemma:Slits}, the curves $\gamma$ and $\delta$ should be regarded as representatives of classes in
\[
H_1(T\setminus\{\zeta_1,\zeta_2\}).
\]
If one of the chosen representatives passes through $\zeta_1$ or $\zeta_2$, we replace it by a sufficiently small perturbation in $T\setminus\{\zeta_1,\zeta_2\}$. Such a perturbation does not change the modulo ${2}$ intersection computations below.}

Without loss of generality, we assume that the lift of $I$ under $P$ has left endpoint at the origin and right endpoint at $(\length(I),0)$.
We denote the modulo ${2}$ intersection numbers of $I$ with the chosen representatives of $\gamma$ and $\delta$ by
\[
i(I,\gamma)
\qquad\text{and}\qquad
i(I,\delta).
\]
By \Cref{lemma:Slits}, these intersection numbers determine the relevant relative homology class modulo $2$. That is, given $I$ and $I'$, if 
\[
\begin{split}
i(I,\delta) = i (I',\delta)
\quad
\text{and}
\quad
i(I,\gamma) = i (I',\gamma)
\end{split}
\]
then $[I] = [I']$ as classes in $H_1(T, \{\zeta_1,\zeta_2\};\Z/2\Z)$.
\textcolor{black}{%
See \Cref{fig:IntersectionNumbers} for an example after applying the small perturbation described above.}

\textcolor{black}{%
Let $N=\frac{\length(I)}{a}$. As in the proof of \Cref{lemma:RelSubsequece}, the number of intersections of $I$ with $\gamma$ is determined by the number of fundamental horizontal periods crossed by $I$, while the number of intersections with $\delta$ is determined by the number of diagonal periods crossed by $I$. Hence
\[
i(I,\gamma)=\lfloor N\rfloor
\qquad
\text{and}
\qquad
i(I,\delta)=\lfloor N\alpha\rfloor.
\]
}

\begin{figure}
\centering
\includegraphics{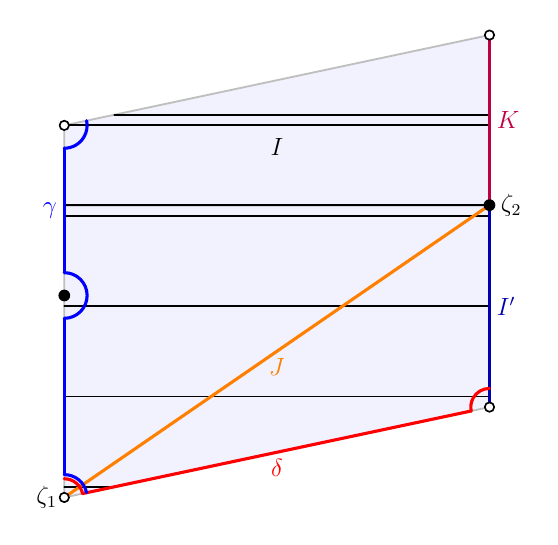}
\caption{The curves $\gamma$ and $\delta$ are representatives of generators of $H_1(T)$. The curves $I$ and $I'$ represent the same class in $H_1(T,\Sigma;\Z/2\Z)$ because the intersection numbers with $\gamma$ and $\delta$ have the same parity, \ie $i(I,\gamma) \equiv i(I'
,\gamma)$ and $i(I,\delta) \equiv i(I',\delta)  \mod 2$. The segments $I'$, $J$, and $K$ represent three different classes. Also, $T \#_IT = T\#_{I'}T$.}
\label{fig:IntersectionNumbers}
\end{figure}

We split the proof into different cases according to the length of the horizontal slit $I$ of $M$:

\begin{itemcases}
\item \label{case:N_even}
$N=2k$ for some positive integer $k$.

\item \label{case:N_odd}
$N=2k+1$ for some non-negative integer $k$.

\item \label{case:N_between}
$2k<N<2k+1$ or $2k+1<N<2k+2$.
\end{itemcases}

\textcolor{black}{\emph{Proof of \Cref*{case:N_even}.}}
Suppose that $N=2k$ for some positive integer $k$. Thus $I$ has even intersection number with a representative of $\gamma$.

We now choose a segment $I'$ joining the same endpoints as $I$ and satisfying
\[
[I']=[I]
\quad\text{in}\quad
H_1(T,\Sigma;\mathbb Z/2\mathbb Z).
\]
There are two natural vertical choices. These choices are vertical because
\[
\hol_T(I)=(\length(I),0)=(2ka,0)
\]
and
\[
\hol_T(\gamma)=(a,\alpha).
\]
Thus, after subtracting $2k$ copies of $\gamma$, the remaining displacement is vertical. Equivalently, an alternative representative $I'$ can be chosen with vertical holonomy
\[
\hol_T(I')=(0,2k\alpha)
\mod \Lambda.
\]

Namely, $I'$ can be chosen to be the image under $P$ of one of the following two vertical segments in $\mathbb R^2$:
\begin{enumerate}[label=(\roman*)]
\item the segment joining $(0,1)$ to $(0,1-\{N\alpha\})$;
\item the segment joining $(0,1-\{N\alpha\})$ to $(0,0)$.
\end{enumerate}
Each of these vertical segments has intersection number $0$ with $\gamma$ modulo $2$, and
\[
i(I,\gamma)=2k=0 \pmod 2.
\]
The choice between (i) and (ii) is made so that the intersection number with $\delta$ also agrees with that of $I$. Since
\[
i(I,\delta)=\lfloor N\alpha\rfloor,
\]
we choose the segment in (i) when $\lfloor N\alpha\rfloor$ is even, and the segment in (ii) when $\lfloor N\alpha\rfloor$ is odd. With this choice, we have
\[
i(I',\gamma)=i(I,\gamma)
\quad\text{and}\quad
i(I',\delta)=i(I,\delta)
\qquad \pmod 2.
\]
Therefore, by \Cref{lemma:Slits},
\[
T_1\#_I T_2=T_1\#_{I'}T_2.
\]

For example, in \Cref{fig:IntersectionNumbers}, the segment $K$ corresponds to the choice in (i), while the segment labeled $I'$ corresponds to the choice in (ii). Both have intersection number $0$ with $\gamma$ modulo $2$. The segment $K$ has intersection number $0$ with $\delta$, while the segment $I'$ has intersection number $1$ with $\delta$. Thus, depending on the parity of $\lfloor N\alpha\rfloor$, one of these two choices gives the required representative.

We now estimate the area exchange between $I$ and $I'$. The two slits $I$ and $I'$ divide the torus into rectangles with horizontal and vertical sides. Since $N=2k$, a vertical closed curve crossing the horizontal slit meets $2k$ horizontal strips. Thus the coloring contains $k$ horizontal rectangles of each color.

Assume that
\[
q_l\leq N<q_{l+1}.
\]
By the best-approximation property of $q_l$, each horizontal rectangle appearing in this decomposition has height at least
\[
\lVert q_l\alpha\rVert .
\]
Each such rectangle has width $a$. Hence each of the two colored regions has area at least
\[
k\,a\,\lVert q_l\alpha\rVert .
\]
Since $q_l\leq N=2k$, we have $k\geq q_l/2$. Therefore each of the two regions has area at least
\[
\frac{a q_l}{2}\lVert q_l\alpha\rVert .
\]
If $\alpha$ is badly approximable, then there exists $c>0$ such that
\[
q_l\lVert q_l\alpha\rVert \geq c
\]
for every $l$. Consequently, each of the two regions has area at least
\[
c_1:=\frac{ac}{2}.
\]
Thus both numbers defining the area exchange are bounded below by a positive constant independent of $N$.

\emph{Proof of \Cref*{case:N_odd}.}
Suppose that $N=2k+1$. We can choose $I'$ to be one of the following:
\begin{enumerate}[label=(\roman*)]
\item the $P$-image of the shortest segment in $\mathbb R^2$ joining the points $(0,1)$ and $(a,1-\{(N-1)\alpha\})$; or
\item the $P$-image of the shortest segment in $\mathbb R^2$ joining the points $(0,0)$ and $(a,1-\{(N-1)\alpha\})$.
\end{enumerate}
These are the relevant choices because, after fixing a representative of $\gamma$ in $H_1(T)$, for example the curve shown in \Cref{fig:IntersectionNumbers}, each of these two choices intersects this representative exactly once. On the other hand, the horizontal segment $I$ intersects the representative of $\gamma$ exactly $2k+1$ times.

Of these two choices, exactly one gives the desired segment $I'$. Both choices have the same intersection number with $\gamma$ modulo $2$, so the correct choice is determined by the intersection number with $\delta$. After fixing a representative of $\delta$, as in \Cref{fig:IntersectionNumbers}, the first segment does not intersect $\delta$, while the second segment intersects $\delta$ exactly once. Recall that $I$ intersects the representative of $\delta$ exactly $\lfloor N\alpha\rfloor$ times. Thus the choice of $I'$ must have intersection number $\lfloor N\alpha\rfloor$ with $\delta$ modulo $2$.

The horizontal curve $I$ intersects the vertical closed curve $\gamma$ exactly $N$ times. The segments $I$ and $I'$ separate $T$ into two regions, which we color with two colors. Together with $\gamma$, they determine $N$ rectangles. Some of these rectangles are cut by $I'$, and have different colors on the two sides of the cut.

In cases (i) and (ii), the slope of $I'$ is
\[
-\frac{\{(N-1)\alpha\}}{a}
\qquad\text{and}\qquad
\frac{1-\{(1-N)\alpha\}}{a},
\]
respectively. We prove the case in which $I'$ has negative slope; the other case is analogous.

One of the colored regions has $\frac{N+1}{2}$ trapezoids and rectangles below $I'$, while the other region has $\frac{N-1}{2}$ trapezoids and rectangles below $I'$. 

Assume that
\[
q_l\leq N<q_{l+1}.
\]
Then the smallest height of these trapezoids or rectangles is at least $\lVert q_l\alpha\rVert$.

Let $M$ be the number of rectangles cut by $I'$, where $2\leq M\leq N-1$. There are precisely $\frac{M}{2}$ such rectangles in each region. 
In \Cref{fig:ChoiceofSlitEvencase}, $M=2$.

\begin{figure}
\centering\includegraphics[scale=0.85]
{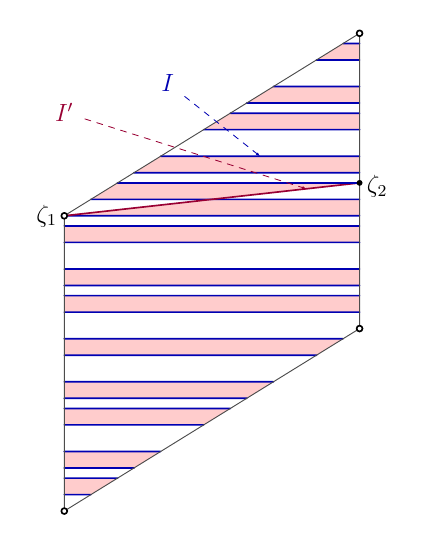}
\caption{The torus is represented by a fundamental parallelogram with opposite sides identified. The horizontal slit $I$ starts at the point labeled $\zeta_1$ and ends at the point labeled $\zeta_2$. Its length satisfies $16<N=\length(I)/a<17$. The long horizontal slit $I$ and the shorter non-horizontal slit $I'$, both joining $\zeta_1$ to $\zeta_2$, represent the same class in $H_1(T,\Sigma;\mathbb Z/2\mathbb Z)$. The red regions represent one of the two regions separated by the slits. Here $\alpha=(-1+\sqrt{5})/2=[0;1,\dots]$.}
\label{fig:ChoiceofSlitEvencase}
\end{figure}

The number of strips not cut by $I'$ is exactly $N-M$. One of the regions has $\frac{N-M+1}{2}$ such strips of full length $a$, and the other has $\frac{N-M-1}{2}$ such strips of full length. For the strips cut by $I'$, the following list gives lower bounds for the lengths of the consecutive bases of the trapezoids:
\[
\left\{
0,
a\frac{\lVert q_l\alpha\rVert}{\{(N-1)\alpha\}},
a\frac{2\lVert q_l\alpha\rVert}{\{(N-1)\alpha\}},
\dots,
a\frac{M\lVert q_l\alpha\rVert}{\{(N-1)\alpha\}}
\right\}.
\]

The area of one of the regions is bounded below by
\[
\left[
2a\sum_{i=0}^{\frac{M-2}{2}}
\frac{\lVert q_l\alpha\rVert}{2}
\left(
\frac{2i\lVert q_l\alpha\rVert}{\{(N-1)\alpha\}}
+
\frac{(2i+1)\lVert q_l\alpha\rVert}{\{(N-1)\alpha\}}
\right)
\right]
+
a\frac{N-M+1}{2}\lVert q_l\alpha\rVert .
\]

The analogous lower bound for the other region is
\begin{equation}
\left[
2a\sum_{i=1}^{\frac{M}{2}}
\frac{\lVert q_l\alpha\rVert}{2}
\left(
\frac{(2i-1)\lVert q_l\alpha\rVert}{\{(N-1)\alpha\}}
+
\frac{(2i)\lVert q_l\alpha\rVert}{\{(N-1)\alpha\}}
\right)
\right]
+
a\frac{N-M-1}{2}\lVert q_l\alpha\rVert .
\label{eq:LowerBoundII}
\end{equation}

For the first region, we obtain the lower bound
\begin{equation}
\begin{split}
a\frac{M^2-M}{2}\lVert q_l\alpha\rVert^2
\frac{1}{\{(N-1)\alpha\}}
+
a\frac{N-M+1}{2}\lVert q_l\alpha\rVert
&\geq
a\frac{M^2-M}{2}\lVert q_l\alpha\rVert^2 \\
&\quad+
a\frac{N-M+1}{2}\lVert q_l\alpha\rVert .
\end{split}
\label{eq:LowerBoundI}
\end{equation}
Fix $t\in(0,1)$. We consider two cases:
\begin{enumerate}[label=(\roman*)]
\item If $tN>M$, then a lower bound for \Cref{eq:LowerBoundI} is
\begin{equation}
\begin{split}
a\frac{N-tN+1}{2}\lVert q_l\alpha\rVert
&\geq
a(1-t)\frac{N}{2}\lVert q_l\alpha\rVert \\
&\geq
a(1-t)\frac{q_l}{2}\lVert q_l\alpha\rVert
\geq
\frac{a(1-t)c}{2}.
\end{split}
\label{eq:area1case1}
\end{equation}

\item If $tN\leq M$, then, using the additional fact that $M\leq N-1$, a lower bound for \Cref{eq:LowerBoundI} is
\begin{equation}
\begin{split}
a\frac{M^2-M}{2}\lVert q_l\alpha\rVert^2
&\geq
a\frac{(tN)^2+1-N}{2}\lVert q_l\alpha\rVert^2 \\
&\geq
a\frac{(tN-1)^2}{2}\lVert q_l\alpha\rVert^2 \\
&\geq
a\left(
\frac{tq_{l-1}}{2}\lVert q_l\alpha\rVert
\right)^2
\geq
\frac{at^2c'^2c^2}{4}.
\end{split}
\label{eq:area1case2}
\end{equation}
\end{enumerate}
In the last line of \Cref{eq:area1case2}, we used \Cref{BigThmContFrac}, \cref{baddlyapproxc}. In both cases, the lower bounds are uniform and independent of $N$.

For the second region, a lower bound for \Cref{eq:LowerBoundII} is
\begin{equation}
a(M^2+M)\frac{\lVert q_l\alpha\rVert^2}{2}
+
a(N-M-1)\frac{\lVert q_l\alpha\rVert}{2}.
\label{eq:LowerBoundII1}
\end{equation}
Using the same ideas as above, for a fixed $t\in(0,1)$, we obtain two lower bounds:
\begin{enumerate}[label=(\roman*)]
\item If $M>tN$, then \Cref{eq:LowerBoundII1} is bounded below by
\begin{equation}
a\left(
t q_l\frac{\lVert q_l\alpha\rVert}{2}
\right)^2
\geq
\frac{at^2c^2}{4}.
\label{eq:area2case1}
\end{equation}

\item If $M\leq tN$, then \Cref{eq:LowerBoundII1} is bounded below by
\begin{equation}
a(1-t)q_{l-1}\frac{\lVert q_l\alpha\rVert}{2}
\geq
\frac{a(1-t)c'c}{2},
\label{eq:area2case2}
\end{equation}
\end{enumerate}
where we again use \Cref{BigThmContFrac}, \cref{baddlyapproxc}. This shows that the areas of both regions are bounded below by $c_2$, the minimum of the expressions in \Cref{eq:area1case1,eq:area1case2,eq:area2case1,eq:area2case2}.

\emph{Proof of \Cref{case:N_between}.}
Finally, we consider the following two subcases:

\begin{subitemcases}
\item \label{case:N_between_3a}
$2k<N<2k+1$.

\item \label{case:N_between_3b}
$2k+1<N<2k+2$.
\end{subitemcases}

Proof of \Cref{case:N_between_3a}.
Suppose that $2k<N<2k+1$. We can choose $I'$ to be the image under $P$ of either the segment with endpoints
\[
(0,0)
\quad\text{and}\quad
(a(N-2k),1-\{2k\alpha\}),
\]
or the segment with endpoints
\[
(0,1)
\quad\text{and}\quad
(a(N-2k),1-\{2k\alpha\}).
\]
The closed curve determined by $I$ and $I'$ separates $T$ into two regions, which we denote by $A_1$ and $A_2$. Define $S_1\subset T$ by
\[
S_1 =
\{p\in T:\text{ there exists }(x,y)\in P^{-1}(p)
\text{ with }0\leq x\leq a(N-2k)\},
\]
and let
\[
S_2=T\setminus S_1.
\]

On $S_1$, the area exchange is bounded below using the method from \Cref{case:N_odd}. On $S_2$, the area exchange is bounded below using the method from \Cref{case:N_even}. Note that the rectangles in $S_1$ have width $as$, while the rectangles in $S_2$ have width $a(1-s)$, where
\[
s=N-2k.
\]
Therefore,
\begin{equation}
\begin{split}
B_1:=\area(A_1)
&=\area(A_1\cap S_1)
+\area(A_1\cap S_2) \\
&\geq c_2s+c_1(1-s)\\
&\geq \min\{c_1,c_2\}
=:c_3.
\end{split}
\end{equation}
Similarly,
\[
B_2:=\area(A_2)\geq c_3.
\]
The constant $c_3$ is independent of the lengths of $I$ and $I'$.

Proof of \Cref{case:N_between_3b}.
Suppose that $2k+1<N<2k+2$. We can choose $I'$ to be the image under $P$ of either the segment with endpoints
\[
(a,\alpha+1)
\quad\text{and}\quad
(a(N-2k-1),\alpha+1-\{(2k+2)\alpha\}),
\]
or the segment with endpoints
\[
(a,\alpha)
\quad\text{and}\quad
(a(N-2k-1),\alpha+1-\{N\alpha\}).
\]
As above, the closed curve determined by $I$ and $I'$ separates $T$ into two regions, which we denote by $A_1$ and $A_2$. Let $S_1\subset T$ be the set
\[
S_1 =
\{p\in T:\text{ there exists }(x,y)\in P^{-1}(p)
\text{ with }ar\leq x\leq a\},
\]
where
\[
r=N-2k-1.
\]
Then, using the estimates from \Cref{case:N_even,case:N_odd} as above, we obtain
\[
\area(A_i)\geq c_1r+c_2(1-r)
\geq \min\{c_1,c_2\}
=c_3
\]
for $i=1,2$. Thus, in this case as well, both numbers defining the area exchange are bounded below by a positive constant independent of the lengths of $I$ and $I'$.
\end{proof}

We will use the following statement about the convergence of tremored surfaces in the proof of \Cref{thm:NonRecurrenceofRel}.

\begin{PROP}
Let $M,M_1,M_2,\dots\in\E$ be surfaces whose horizontal flows are aperiodic and non-minimal. Let
\[
\beta_i=\beta_{v_i}
\]
be foliation cocycles with uniformly bounded total variation; that is, there exists $a>0$ such that for all $i$
\[
|L|_{M_i}(\beta_i)\leq a.
\]
Assume that the signed mass is constant along the sequence, \ie
\[
L_{M_i}(\beta_i)=c.
\]
If
\[
\trem_{\beta_i}M_i
\longrightarrow
\trem_\beta M
\]
for some $\beta$ of signed mass $c$, then
\[
\beta_i
\longrightarrow
\beta
\qquad\text{and}\qquad
M_i
\longrightarrow
M.
\]
\label{prop:convergence}
\end{PROP}

\begin{RMK}
The converse holds and is a particular case of the more general \cite[Proposition 4.7]{2020tremors}. 
\end{RMK}

Before the proof, we will introduce 
\Cref{lemma:UOrbitAndTremor} and \Cref{lemma:compactsequences}, which will be used in the argument of the proof of \Cref{prop:convergence}.

\begin{LEM}
\label{lemma:UOrbitAndTremor}
Let $M,N \in \E$ be horizontally aperiodic and non-minimal surfaces. Let $\beta$ and $\beta'$ be foliation cocycles on $M$ and $N$, respectively. If
\[
\trem_\beta M = \trem_{\beta'} N,
\]
then $M =u_sN$ for some $s \in \R.$ In particular 
\[
\beta = \beta' + s\d y.
\]
\end{LEM}

\begin{proof}
Since $\trem_\beta M = \trem_{\beta'} N$, we can write
\[
N = \trem_{\beta-\beta'}M
\]
by identifying $\beta' \in \tremspace_N$ with some foliation cocycle in $\tremspace_M$. 
We denote the identified foliation cocycle on $M$ by $\beta'$ to avoid introducing more notation. 
This was formalized in \cite[Proposition 6.1]{2020tremors} using the tremor comparison homeomorphism, or TCH. 
In \Cref{lemma:AperiodicNonMinimal}, we saw that $M$ can be written as a slit construction $T_1 \#_I T_2$ along an embedded horizontal slit $I$, where $T_1$ and $T_2$ are two identical copies of a torus $T$. The tremored surface $\trem_{\beta-\beta'}M$ can be represented by the slit construction
\[
(u_{s_1}T_1) \#_I (u_{s_2}T_2)
\]
where $s_1,s_2$ are real numbers and $I$ is still an embedded horizontal segment on each $u_{s_i}T_i$. This is because the action of the matrix $u_s$ sends horizontal segments to horizontal segments.
\textcolor{black}{
\begin{CLAIME}
Since $T$ is horizontally aperiodic, we have $s_1=s_2$.
\end{CLAIME}
}
\begin{proof}
\renewcommand{\qedsymbol}{$\triangle$}
For a translation surface $X$, let
\[
\operatorname{Per}(X)
:=
\left\{
\hol_X(\eta): \eta\in H_1(X;\mathbb Z)
\right\}
\subset \mathbb R^2
\]
denote its group of absolute periods.

Suppose, by contradiction, that
\[
s_3:=s_2-s_1\neq 0.
\]
Since $N\in\E$, and
\[
N=(u_{s_1}T_1)\#_I(u_{s_2}T_2)
=
u_{s_1}\bigl(T_1\#_I(u_{s_3}T_2)\bigr),
\]
the $U$-invariance of $\E$ implies that
\[
T_1\#_I(u_{s_3}T_2)\in\E.
\]

Let $\Lambda$ be the lattice defining the torus $T$. Since
\[
T_1=T_2=T,
\]
the absolute periods coming from $T_1$ are precisely $\Lambda$, while the absolute periods coming from $u_{s_3}T_2$ are precisely $u_{s_3}\Lambda$. Hence
\[
\operatorname{Per}\bigl(T_1\#_I(u_{s_3}T_2)\bigr)
\supset
\Lambda+u_{s_3}\Lambda .
\]

On the other hand, since
\[
T_1\#_I(u_{s_3}T_2)\in\E,
\]
this surface is obtained by gluing two copies of a single torus. Thus there is a lattice $\Lambda'$ such that
\[
\operatorname{Per}\bigl(T_1\#_I(u_{s_3}T_2)\bigr)=\Lambda'.
\]

We now show that this is impossible when $s_3\neq 0$. Indeed, for every $v = (x,y) \in \Lambda$, we have
\begin{equation}
\label{eq:lattice_trick}
(s_3y,0) =u_{s_3}v - v
\in \Lambda+u_{s_3}\Lambda \subset 
\Lambda'.
\end{equation}

Since $T$ is horizontally aperiodic, $\Lambda$ does not contain nonzero horizontal vectors. However, there exists a sequence of vectors \[ \{(x_n,y_n)\}_{n\in\N}\subset \Lambda \] such that, for all $n\in\N$, $y_n\neq 0$, and \[ \lim_{n\to\infty} y_n=0. \] Therefore, \Cref{eq:lattice_trick}
implies that
\[
\{(s_3y_n,0)\}_{n \in \N} \subset \Lambda'
\]
is a sequence of nonzero vectors accumulating at the origin. This is a contradiction, since $\Lambda'$ is uniformly discrete. Thus $s_3=0$, and therefore $s_1=s_2$. 
\end{proof}
The claim proves that $N=u_sM$ for some $s\in\mathbb R$. Since $\trem_{s\,\d y}M=u_sM$, and since \[ \trem_{\beta}M=\trem_{\beta'+s\,\d y}M, \] we have \[ \beta=\beta'+s\,\d y, \] because the surface $M$ is horizontally aperiodic.
\end{proof}

\begin{LEM}
Under the assumptions of \Cref{prop:convergence}, if
\[
\trem_{\beta_i}M_i
\longrightarrow
\trem_{\beta}M,
\]
then the sequence $\{M_i\}_{i\geq 1}$ is contained in a compact set.
\label{lemma:compactsequences}
\end{LEM}

\begin{proof}
Suppose, by contradiction, that $M_i\to\infty$. Then there exists a sequence of saddle connections $\gamma_1,\gamma_2,\dots$
in $M_1,M_2,\dots$, respectively, such that
\[
\lVert \hol_{M_i}(\gamma_i)\rVert\longrightarrow 0.
\]

\textcolor{black}{
For each $i$, let
\[
\widehat{M_i}:=[(f_i,M_i)]
\]
be a marking of $M_i$, and let
\[
\widehat{\trem_{\beta_i}M_i}:=[(g_i,\trem_{\beta_i}M_i)]
\]
be a marking of $\trem_{\beta_i}M_i$. Define
\[
\widehat{\gamma_i}:=f_i^{-1}(\gamma_i).
\]
By \Cref{eq:holtrem},
\[
\begin{split}
\hol_{\widehat{\trem_{\beta_i}M_i}}(\widehat{\gamma_i})
&=
\left(
\holx_{\widehat{M_i}}(\widehat{\gamma_i})
+(f_i^*\beta_i)(\widehat{\gamma_i}),
\holy_{\widehat{M_i}}(\widehat{\gamma_i})
\right) \\
&=
\left(
\holx_{M_i}(\gamma_i)+\beta_i(\gamma_i),
\holy_{M_i}(\gamma_i)
\right).
\end{split}
\]
}
Since the cocycles $\beta_i$ have uniformly bounded total variation, there exists $a>0$ such that
\[
|L|_{M_i}(\beta_i)\leq a
\]
for all $i$. Therefore
\[
\beta_i(\gamma_i)\longrightarrow 0.
\]
Since
\[
\lVert \hol_{M_i}(\gamma_i)\rVert\longrightarrow 0,
\]
it follows that
\[
\lVert
\hol_{\widehat{\trem_{\beta_i}M_i}}(\widehat{\gamma_i})
\rVert
\longrightarrow 0.
\]
Hence the sequence $\{\trem_{\beta_i}M_i\}_{i\in\N}$ is not contained in a compact set. This contradicts the assumption that
\[
\trem_{\beta_i}M_i
\longrightarrow
\trem_{\beta}M.
\]
Therefore $\{M_i\}_{i\geq 1}$ is contained in a compact set.
\end{proof}

\begin{proof}[Proof of \Cref{prop:convergence}]
Assume that
\[
\trem_{\beta_i}M_i
\longrightarrow
\trem_\beta M.
\]
We will prove that
\[
\beta_i
\longrightarrow
\beta
\qquad\text{and}\qquad
M_i
\longrightarrow
M.
\]

By \Cref{lemma:compactsequences}, the sequence $\{M_i\}$ is contained in a compact set. Hence, after passing to a subsequence $\{M_{i_k}\}$, we may assume that there exists $N\in\E$ such that
\[
M_{i_k} \longrightarrow N.
\]
The foliation cocycles $\beta_i=\beta_{v_i}$ correspond to transverse measures $v_i$. These measures are uniformly bounded because the total variations of the foliation cocycles are uniformly bounded; that is, there exists $a>0$ such that for all $i$
\[
|L|_{M_i}(\beta_i)\leq a.
\]
Therefore, after passing to a further subsequence, we may assume that
\[
\beta_{i_k}
\longrightarrow
\beta'
\]
for some foliation cocycle $\beta'$ on $N$. To avoid introducing additional notation, we denote this further subsequence again by $\{\beta_{i_k}\}$.

Thus we have obtained subsequences $\{M_{i_k}\}$ and $\{\beta_{i_k}\}$ such that
\[
M_{i_k} \longrightarrow N
\qquad\text{and}\qquad
\beta_{i_k} \longrightarrow \beta'.
\]

By the continuity of the tremor map, proved in \cite[Proposition 4.7]{2020tremors}, we have
\[
\trem_{\beta_{i_k}}M_{i_k} \longrightarrow \trem_{\beta'}N.
\]
On the other hand, by hypothesis,
\[
\trem_{\beta_i}M_i \longrightarrow \trem_\beta M.
\]
Therefore
\[
\trem_\beta M=\trem_{\beta'}N.
\]
By \Cref{lemma:UOrbitAndTremor}, it follows that
\[
\beta=\beta'+s\,\d y
\]
for some $s\in\R$.

Since each $\beta_{i_k}$ has signed mass equal to $c$, and since
\[
M_{i_k} \longrightarrow N
\qquad\text{and}\qquad
\beta_{i_k} \longrightarrow \beta',
\]
continuity of the map $(\eta,X)\mapsto L_X(\eta)$ implies that
\[
L_N(\beta')
=
\lim_{k\to\infty}L_{M_{i_k}}(\beta_{i_k})
=
c.
\]

Now compute the signed masses of $\beta$ and $\beta'$. Since $\beta$ has signed mass $c$, we have
\[
c
=
L_M(\beta)
=
L_N(\beta')+L_N(s\,\d y)
=
c+s\area(N).
\]
Hence $s=0$. Therefore, by \Cref{lemma:UOrbitAndTremor},
\[
\beta=\beta'
\qquad\text{and}\qquad
M=N.
\]

We have shown that every convergent subsequence of $\{M_i\}$ and $\{\beta_i\}$ has the same limit, namely $M$ and $\beta$, respectively. Since these sequences are contained in compact sets, it follows that
\[
M_i \longrightarrow M
\qquad\text{and}\qquad
\beta_i \longrightarrow \beta.
\]
\end{proof}

\begin{proof}[Proof of \Cref{thm:NonRecurrenceofRel}]
Write
\[
M=T_1\#_I T_2
\]
with $I$ horizontal. We first reduce to the case where the horizontal slit $I$ has length $a$.

\textcolor{black}{We claim that it is enough to prove the following normalized statement:} if $M=T_1\#_I T_2$, where $I$ is horizontal and has length $a$, then there exists $C>0$ such that, for every sequence of positive numbers $t_i\to\infty$,
\begin{equation}
d\bigl(\trem_\beta M,\Rel_{t_i}\trem_\beta M\bigr)\geq C.
\label{eq:KeyEquation}
\end{equation}

\textcolor{black}{
Indeed, suppose that the normalized statement has been proved. Let
\[
M'=T_1\#_{I'}T_2
\]
be a surface in the same real Rel ray, where $I'$ is horizontal but has arbitrary length. Choose $s\in\mathbb R$ such that
\[
M=\Rel_s M'=T_1\#_I T_2
\]
has horizontal slit $I$ of length $a$.\\
We prove the \namecref{thm:NonRecurrenceofRel} for $M'$. Arguing by contradiction, suppose that there exists a sequence of positive numbers $t_i\to\infty$ such that
\[
\trem_\beta M'
=
\lim_{i\to\infty}\Rel_{t_i}\trem_\beta M'.
\]
By the continuity of $\Rel_s(\cdot)$ at $\trem_\beta M'$, see \cite[Proposition 4.3]{HoroDynamics}, we obtain
\[
\begin{split}
\trem_\beta M
&=
\trem_\beta\Rel_s M' \\
&=
\Rel_s\trem_\beta M' \\
&=
\lim_{i\to\infty}\Rel_s\Rel_{t_i}\trem_\beta M'.
\end{split}
\]
Here we used that tremors and real Rel commute, namely
\[
\trem_\beta\Rel_s M'=\Rel_s\trem_\beta M',
\]
by \cite[Proposition 6.7]{2020tremors}.\\
Finally, by the flow property of real Rel flow, see \cite[Proposition 4.5]{HoroDynamics}, we have
\[
\begin{split}
\Rel_s\Rel_{t_i}\trem_\beta M'
= &
\Rel_{t_i}\Rel_s\trem_\beta M'\\
= &
\Rel_{t_i}\trem_\beta\Rel_s M'\\
= &
\Rel_{t_i}\trem_\beta M.
\end{split}
\]
Therefore
\[
\trem_\beta M
=
\lim_{i\to\infty}\Rel_{t_i}\trem_\beta M,
\]
which contradicts \Cref{eq:KeyEquation}. Hence the theorem for $M'$ follows from the normalized statement.
}

It remains to prove the normalized statement. Thus, for the rest of the proof, we assume that
\[
M=T_1\#_I T_2,
\]
where $I$ is horizontal and has length $a$.

Assume, by contradiction, that there exists a sequence $t_i\to\infty$ such that
\[
\trem_{\beta}M
=
\lim_{i\to\infty}\Rel_{t_i}\trem_{\beta}M.
\]
For each $i$, we can write
\[
\Rel_{t_i}M=T_1\#_{I_i}T_2=T_1\#_{I'_i}T_2,
\]
where $I_i$ is long and horizontal, while $I'_i$ is the short representative given by \Cref{lemma:short_slit_representative}.

In the notation of \Cref{prop:convergence}, set
\[
M_i:=\Rel_{t_i}M.
\]
Let $\beta_i$ be the unique foliation cocycle on $\Rel_{t_i}M$ such that
\[
\Rel_{t_i}\trem_\beta M
=
\trem_{\beta_i}\Rel_{t_i}M.
\]
Since $\beta$ is non-atomic, each $\beta_i$ is non-atomic. Moreover, for all $i$,
\[
L_{M_i}(\beta_i)=L_M(\beta)=c,
\]
and there exists a constant $d>0$ such that
\[
|L|_{M_i}(\beta_i)\leq d
\qquad\text{and}\qquad
|L|_M(\beta)\leq d.
\]

By \Cref{prop:convergence}, we have
\[
\beta_i\to\beta
\qquad\text{and}\qquad
M_i\to M.
\]
The convergence $M_i\to M$ implies that we can choose the short segments $I'_i$ on $T$ so that
\[
I'_i\to I.
\]
We will now show that the convergence
\[
\beta_i\to\beta
\]
contradicts the assumption that $\alpha$ is badly approximable.

\textcolor{black}{By \Cref{eq:decomp_cocycle}, we may write
\[
\beta=a_1\d y|_{T_1}+a_2\d y|_{T_2},
\]
with $a_1\neq a_2$.}

\textcolor{black}{
We choose $\gamma_1$ to be a vertical closed curve contained in $T_1$ and passing through the first singularity of $M$. Then
\[
\beta(\gamma_1)=a_1.
\]
Similarly, for each $i$, we choose $\gamma_{1,i}$ to be a vertical closed curve contained in $T_1$ for the representation
\[
M_i=T_1\#_{I'_i}T_2
\]
and passing through the first singularity of $M_i$.
}

We will show that
\[
\beta_i(\gamma_{1,i})\not\longrightarrow \beta(\gamma_1)=a_1,
\]
contradicting the convergence $\beta_i\to\beta$.
The key point is that, since $\alpha$ is badly approximable, the area exchange between $I_i$ and $I'_i$ is bounded below by a positive constant independent of $t_i$.

Denote by $B^i_1$ and $B^i_2$ the two numbers defining the area exchange between $I_i$ and $I'_i$. 
\textcolor{black}{
The value of $\beta_i(\gamma_{1,i})$ can be computed from the area exchange between the two slit presentations
\[
M_i=T_1\#_{I_i}T_2
\qquad\text{and}\qquad
M_i=T_1\#_{I'_i}T_2.
\]
Indeed, in the presentation $T_1\#_{I_i}T_2$, the cocycle is given by
\[
\beta=a_1\d y|_{T_1}+a_2\d y|_{T_2}.
\]
After cutting and regluing along $I'_i$, the vertical curve $\gamma_{1,i}$ in the first torus of the presentation
\[
M_i=T_1\#_{I'_i}T_2
\]
passes through the two regions determined by the area exchange. The proportions of this curve lying in the regions carrying the coefficients $a_1$ and $a_2$ are
\[
\frac{B^i_1}{\area(T)}
\qquad\text{and}\qquad
\frac{B^i_2}{\area(T)}.
\]
Therefore, in the period coordinates adapted to the presentation $T_1\#_{I'_i}T_2$, either
\[
\beta_i(\gamma_{1,i})
=
\frac{B^i_1}{\area(T)}a_1
+
\frac{B^i_2}{\area(T)}a_2
\]
or
\[
\beta_i(\gamma_{1,i})
=
\frac{B^i_2}{\area(T)}a_1
+
\frac{B^i_1}{\area(T)}a_2.
\]
This depends on the choice of labels for the colored regions. Without loss of generality, assume that the former equation holds. The proof in the case where the latter equation holds is analogous.
}
Then
\begin{equation}
\begin{split}
|\beta_i(\gamma_{1,i})-\beta(\gamma_1)|
&=
\left|
\frac{B^i_1}{\area(T)}a_1
+
\frac{B^i_2}{\area(T)}a_2
-
a_1
\right|\\
&=
\left|
\left(\frac{B^i_1}{\area(T)}-1\right)a_1
+
\frac{B^i_2}{\area(T)}a_2
\right|\\
&=
\left|
-\frac{B^i_2}{\area(T)}a_1
+
\frac{B^i_2}{\area(T)}a_2
\right|\\
&=
\frac{B^i_2}{\area(T)}|a_1-a_2|.
\end{split}
\label{eq:DiffCoordinates}
\end{equation}
Here we used that
\[
B^i_1+B^i_2=\area(T).
\]

By \Cref{lemma:boundedarea}, the area exchange is bounded below, namely, there exists $c>0$ such that
\[
B^i_j\geq c
\]
for all $j=1,2$ and all $i$. Therefore \Cref{eq:DiffCoordinates} implies that for all $i$
\begin{equation}\label{eq:bound_on_cocycle_difference}
|\beta_i(\gamma_{1,i})-\beta(\gamma_1)|
\geq
\frac{c}{\area(T)}|a_1-a_2|
\neq 0.
\end{equation}

\textcolor{black}{Since $M_i\to M$, we may choose markings
\[
\widehat{M_i}=[(f_i,M_i)]
\qquad\text{and}\qquad
\widehat{M}=[(f,M)]
\]
such that
\[
\widehat{M_i} \longrightarrow \widehat{M}.
\]
After choosing these markings, let
\[
\eta=f^{-1}(\gamma_1)
\]
be a closed curve in the model surface $S$. The curve $\gamma_{1,i}$ on $M_i$ represents the same homology class as $f_i(\eta)$.}
Therefore, the convergence
\[
\beta_i \longrightarrow \beta
\]
would imply
\[
\beta_i(\gamma_{1,i})
\longrightarrow
\beta(\gamma_1).
\]
However, this is impossible by \eqref{eq:bound_on_cocycle_difference}, giving a contradiction.
\end{proof}

\bibliography{references}
\bibliographystyle{alpha}

\end{document}